\documentclass[11pt,a4paper]{amsart} 

\usepackage[utf8]{inputenc}
\usepackage{xcolor}
\usepackage{graphicx}
\usepackage{parskip}
\usepackage{amssymb, amsthm,amsmath}
\calclayout
\usepackage{tikz}
\usetikzlibrary{cd}
\usetikzlibrary{fit}
\usetikzlibrary{knots}
\usetikzlibrary{positioning}
\usetikzlibrary{arrows.meta}

\usepackage{subfig}
\usepackage[arrow, matrix, curve]{xy} 
\usepackage{hyperref}
\usepackage[capitalise]{cleveref}
\usepackage{enumerate}  
\usepackage{enumitem}
\usepackage[style=alphabetic,sorting=nyt,maxbibnames=7,maxcitenames=5]{biblatex}
\usepackage{booktabs}
\usepackage{appendix}
\renewbibmacro{in:}{}
\usepackage{afterpage}
\usepackage{soul}

\theoremstyle{plain}

\newtheorem{theorem}{Theorem}[section]
\newtheorem{lemma}[theorem]{Lemma}

\newtheorem{corollary}[theorem]{Corollary}
\newtheorem{example}[theorem]{Example}

\newtheorem{proposition}[theorem]{Proposition}
\theoremstyle{definition}

\newtheorem{remark}[theorem]{Remark}
\newtheorem{definition}[theorem]{Definition}

\newcommand{\ord}{\mathrm{ord}}
\newcommand{\Fix}{\mathrm{Fix}}
\newcommand{\Isom}{\mathrm{Isom}}
\newcommand{\lcm}{\mathrm{lcm}}

\newcommand{\N}{\mathbb{N}} 
\newcommand{\Z}{\mathbb{Z}}

\newcommand{\R}{\mathbb{R}}

\newcommand{\phee}{\varphi}

\newcommand{\til}[1]{\widetilde{#1}}
\newcommand{\spans}[1]{\langle {#1} \rangle}

\newcommand{\leer}{\varnothing}

\newcommand{\mc}{\mathcal}

\newcommand{\cc}{\mathrm{cc}_{2}} 
 
\newcommand{\ccm}{\mathrm{cc}_{m}} 
\newcommand{\ccp}{\mathrm{cc}_{p}} 
\newcommand{\cck}{\mathrm{cc}_{k}} 
\newcommand{\ccl}{\mathrm{cc}_{l}} 
\newcommand{\tA}{\mathtt{A}}
\newcommand{\tB}{\mathtt{B}}
\newcommand{\tC}{\mathtt{C}}
\newcommand{\tD}{\mathtt{D}}
\newcommand{\tE}{\mathtt{E}}
\newcommand{\tF}{\mathtt{F}_{4}}
\newcommand{\tG}{\mathtt{G}_{2}}
\newcommand{\tH}{\mathtt{H}}

\newcommand{\tI}{\mathtt{I}_{2}}
\newcommand{\atA}{\tilde{\mathtt{A}}} 
\newcommand{\atX}{\widetilde{\mathtt{X}}}
\newcommand{\atB}{\widetilde{\mathtt{B}}}
\newcommand{\atC}{\widetilde{\mathtt{C}}}
\newcommand{\atD}{\widetilde{\mathtt{D}}}
\newcommand{\atE}{\widetilde{\mathtt{E}}}
\newcommand{\atF}{\widetilde{\mathtt{F}}_{4}}
\newcommand{\atG}{\widetilde{\mathtt{G}}_{2}}
\newcommand{\atI}{\widetilde{\mathtt{I}}_{1}}
\newcommand{\lk}{\mathrm{Lk}}

\title{Involutions in Coxeter groups}

\author[A. Michael]{Anna Michael}
\address{%Anna Michael \\ 
Otto von Guericke University Magdeburg \\ 
Faculty of Mathematics, Universit\"atsplatz 1, 39106 Magdeburg, Germany}
\email{anna.michael@ovgu.de}

\author[Y. Santos]{Yuri Santos Rego}
\address{%Yuri Santos Rego \\ 
University of Lincoln\\ 
Charlotte Scott Research Centre for Algebra, School of Mathematics and Physics\\ 
Brayford Pool, LN6 7TS Lincoln, UK}
\email{ysantosrego@lincoln.ac.uk}

\author[P. Schwer]{Petra Schwer}
\address{%Petra Schwer \\ 
Ruprecht Karls University Heidelberg \\
Institute of Mathematics, Mathematikon, Neuenheimer Feld 205, 69120 Heidelberg, Germany}
\email{schwer@mathi.uni-heidelberg.de}

\author[O. Varghese]{Olga Varghese}
\address{%Olga Varghese \\ 
Heinrich Heine University D\"usseldorf\\
Institute of Mathematics,  Universit\"atsstra{\upshape{\ss}}e 1, 40225, D\"usseldorf, Germany}
\email{olga.varghese@hhu.de}

\date{\today}

\begin{document}

\begin{abstract}
	We combinatorially characterize the number $\mathrm{cc}_2$ of conjugacy classes of involutions in any Coxeter group in terms of  higher rank odd graphs. This notion naturally generalizes the concept of odd graphs, used previously to count the number of conjugacy classes of reflections. We provide uniform bounds and discuss some extremal cases, where the number $\mathrm{cc}_2$ is smallest or largest possible. Moreover, we provide formulae for $\mathrm{cc}_2$ in free and direct products as well as for some finite and affine types, besides computing $\mathrm{cc}_2$ for all triangle groups, and all affine irreducible Coxeter groups of rank up to eleven.
\end{abstract}

\maketitle

\begin{center}
    We dedicate this work to Peter Littelmann on the occasion of his retirement.
\end{center}

%%%%%%%%%%%%%%%%%%%%%%%%%%%%%%%%%%%%%%%%%%%%%%%%%%%%%%%%%%%%%%
\section{Introduction}

As abstract analogs  of reflection groups, every Coxeter group $W$ is generated by a standard generating set $S$ of involutions, that is, elements of order two.  Under Tits' geometric representation the generators in $S$ are mapped to abstract reflections along faces of a cone in $\R^{|S|}$, 
and the set of all such reflections in $W$ contains precisely the $W$-conjugates of the standard generators. So every reflection is involutory.  However, not all involutions are reflections. And not all involutions (nor reflections) are in a same conjugacy class. A given group $W$ may admit more than one generating set $S$ with respect to which it is a Coxeter group. While the set of reflections does depend on the chosen Coxeter system $(W,S)$, the set of involutions does not. 

The following questions are hence natural to ask: 
What can one say about the potential discrepancy between the set of reflections and that of involutions in a given Coxeter group? When are these two sets equal? Is there a combinatorial characterization of conjugacy classes of involutions in terms of the Coxeter diagram of a given system? 

One source of inspiration is the following: The number of conjugacy classes of reflections relates to the number of connected components of the odd graph $\Gamma_{odd}$,  which is obtained from a Coxeter diagram $\Gamma$ by removing all edges with even or $\infty$ labels \cite[Lemma~3.6]{BMMN}. 
We generalize this concept and determine the number of conjugacy classes of involutions in \Cref{thm:CardConjClass} by counting connected components of certain higher rank odd graphs $\Gamma^k$, $1\leq k\leq \vert V(\Gamma)\vert$,  derived from $\Gamma$. The first graph $\Gamma^1$ equals the aforementioned odd graph. For details and definitions see \Cref{sec:OddGraphs}.

While the set of involutions is infinite in case of infinite Coxeter groups, there are always only finitely many  conjugacy classes of involutions. This follows for one from the fact that Coxeter groups are CAT(0) and hence have only finitely many conjugacy classes of finite subgroups \cite{Davis2008}, \cite{BridsonHaefliger1999}. In providing an explicit formula for the number $\mathrm{cc}_2(W_\Gamma)$ of conjugacy classes of involutions in terms of combinatorial statistics based on the defining diagram, \Cref{thm:CardConjClass} gives a direct proof of this fact.

\begin{theorem}[Counting conjugacy classes of involutions]
    \label{thm:CardConjClass}
    Let $\Gamma$ be a Coxeter diagram and, for each $k \geq 1$, write $\Gamma^k$ for its corresponding $k$-odd graph (cf. \cref{def:odd-graphs}).
    Then, the cardinality of the set of connected components of $\Gamma^k$ is equal to the cardinality of the set of conjugacy classes of involutions of rank $k$ in the corresponding Coxeter group $W_\Gamma$. 
    The total number $\mathrm{cc}_2(W_\Gamma)$ of conjugacy classes of involutions in $W_\Gamma$ is equal to 
    \[
    \mathrm{cc}_2(W_\Gamma)=\sum^{\mid V(\Gamma)\mid}_{k=1}\vert\pi_0(\Gamma^k)\vert.
    \]
\end{theorem}
 
As a by-product of the proof of \cref{thm:CardConjClass} (see also \cref{lem:LongestEltsRepresentInvolutions}) one has the following. 

\begin{corollary}[Representing conjugacy classes of involutions]
    \label{cor:conjClassRefinement}
    Let notation be as in \Cref{thm:CardConjClass}. 
    The connected components of $\Gamma^k$ are in natural bijection with the conjugacy classes of involutions of rank $k$.     
    To obtain a representative for such a conjugacy class one may choose any vertex in the corresponding  connected component in $\Gamma^k$ and take the longest element in the parabolic subgroup corresponding to it. 
\end{corollary}  

The proof of \Cref{thm:CardConjClass} can be found in \Cref{sec:mainproofs}. 
Our results thus provide an effective method to describe conjugacy classes of involutions in Coxeter groups. A different algorithm has long been described by Richardson \cite{Richardson_1982}, building upon the works of Howlett \cite{Howlett_1980} and Deodhar \cite{Deodhar_1982}. Both our work and Richardson's use results from Deodhar's seminal paper \cite{Deodhar_1982} as key ingredients. From that point onwards, however, our proofs heavily differ --- Richardson's methods are representation-theoretic while ours are combinatorial in flavour.

We record some structural consequences. The next corollary, proved in \Cref{sec:mainproofs}, highlights two extremal cases, where all involutions are reflections or all involutions are conjugate to one another. The theorem thereafter provides bounds on the number of conjugacy classes and characterizes those Coxeter groups that attain the upper bound. 

\begin{corollary}[The extremal cases] 
\label{cor:unterSchranke}
    Let $\Gamma$ be a Coxeter diagram of rank $n$. 
    \begin{enumerate}
    \item There is only one conjugacy class of involutions in $W_\Gamma$, i.e. $\mathrm{cc}_2(W_\Gamma)=1$,  if and only if the Coxeter diagram $\Gamma$ is complete, the odd graph $\Gamma^1$ is connected, and every edge label in $\Gamma$ is odd or $\infty$.
    \item Every involution in $W_\Gamma$ is a reflection if and only if the diagram $\Gamma$ is complete and every edge label in $\Gamma$ is odd or $\infty$, in which case $\mathrm{cc}_2(W_\Gamma)=\vert\pi_0(\Gamma^1)\vert \leq n$. 
    \end{enumerate}
\end{corollary}

For a first example illustrating the results consider the Coxeter group of type $\atG$ with Coxeter diagram as shown in \Cref{fig:graphsG2}. The odd graph $\Gamma^1$, depicted in the same figure, has two connected components and so does $\Gamma^2$. The graph $\Gamma^3$ is empty and hence,  by \Cref{thm:CardConjClass}, one has  $\mathrm{cc}_2(W_{\atG}) =4$, the total number of connected components in the graphs $\Gamma^1$ and $\Gamma^2$. Moreover, the set of involutions in $\atG$ is strictly bigger than the set of reflections by item~(2) of \Cref{cor:unterSchranke}. 

\begin{figure}[h]
	\begin{center}
	\captionsetup{justification=centering}
		\begin{tikzpicture}[scale=1, transform shape]
            \draw[fill=black] (0, 0) circle (2pt);
            \draw[fill=black] (1, 0) circle (2pt);
            \draw[fill=black] (2, 0) circle (2pt);
            \draw (0,0)--(2, 0);
            \node at (1.5, 0.2) {$6$};
            \node at (0, -0.4) {$s_1$};
            \node at (1, -0.4) {$s_2$};
            \node at (2, -0.4) {$s_3$};
            \node at (1, -1.5) {$\atG$};

            \draw[fill=black] (4, 0) circle (2pt);
            \draw[fill=black] (5, 0) circle (2pt);
            \draw[fill=black] (6, 0) circle (2pt);
            \draw (4,0)--(5, 0);
            \node at (4, -0.5) {$W_{\{1\}}$};
            \node at (5.1, -0.5) {$W_{\{2\}}$};
            \node at (6.3, -0.5) {$W_{\{3\}}$};
            \node at (5, -1.5) {$\Gamma^1$};

            \draw[fill=black] (9, 0) circle (2pt);
            \draw[fill=black] (11, 0) circle (2pt);
            \node at (9, -0.5) {$W_{\{1,3\}}$};
            \node at (11, -0.5) {$W_{\{2,3\}}$};
            \node at (10, -1.5) {$\Gamma^2$};

            \node at (13,-0.25) {$\leer$};
            \node at (13,-1.5) {$\Gamma^3$};

    \end{tikzpicture}
	\caption{$\mathrm{cc}_2(W_{\atG})=\vert\pi_0(\Gamma^1)\vert+\vert\pi_0(\Gamma^2)\vert=4$.}
	 \label{fig:graphsG2}

 \end{center}
\end{figure}

From this example, and keeping in mind that the number of conjugacy classes of reflections is at most the rank $n$ of a Coxeter system, we broadly ask: in what situations is the discrepancy between the involutions and reflections as large (or low) as possible? That is, are there functions $g(n), f(n)$ for which $g(n) \leq \cc(W_\Gamma) \leq f(n)$ for $|V(\Gamma)|=n$? 

For instance, one can observe --- see \Cref{ex:finiteAbelian} --- that finite abelian Coxeter groups of rank $n$ have $2^n - 1$ conjugacy classes of nontrivial involutions, which is much larger than $n$ if $n\geq 2$. It turns out that this is the maximal possible number for a finite Coxeter group of rank $n$. It is hence natural to wonder whether nonabelian groups of rank $n$ also attain this maximum and whether infinite Coxeter groups of rank $n$ can have more conjugacy classes of involutions than $2^n-1$. The answers --- yes and no, respectively --- are part of the content of \Cref{thm:obereSchranke}, which we prove in \Cref{sec:mainproofs}. In the working example of the group of type $\atG$, which has rank three, \Cref{thm:obereSchranke} implies that $\cc(W_{\atG})$ does not attain the upper bound possible with respect to its rank. 

Regarding lower bounds, we have seen in \Cref{cor:unterSchranke} that $\cc$ might well equal $1$, so that the bound $1 \leq \cc(W_\Gamma)$ is sharp. As this is not very informative, we construct from each higher rank odd graph $\Gamma^k$ a variant $\Omega^k$ with same vertex set that easily yields a more precise lower bound, which in many examples becomes an equality.

\begin{theorem}[Bounding the number of conjugacy classes of involutions]
\label{thm:obereSchranke}
Let $\Gamma$ be a Coxeter diagram of rank $n$ and denote by $W_\Gamma$ the associated Coxeter group. Write $\Omega^k$, with $1 \leq k \leq n$, for its $\mc{O}$-graphs (cf. \cref{def:omegas}) and suppose $W_{\Delta_1},\ldots, W_{\Delta_m}$ are its maximal (with respect to inclusion) spherical standard parabolic subgroups. Then the following inequalities hold. 
    \begin{equation} \label{eq:bound1}
    |\pi_0(\Omega^1)| + \ldots + |\pi_0(\Omega^n)| \leq \cc(W_\Gamma)\leq \cc(W_{\Delta_1})+\ldots +\cc(W_{\Delta_m}), \quad \text{ and}
    \end{equation} 
\begin{equation} \label{eq:bound2}
\begin{aligned}
1 \leq \mathrm{cc}_2(W_\Gamma) \leq 
\begin{cases} 
2^n-1 & \text{ if } W_\Gamma \text{ is finite}, \\
2^n-2 & \text{ otherwise.} 
\end{cases}
\end{aligned}
\end{equation}
The numerical bounds in Inequality~\eqref{eq:bound2} are sharp and are attained also by nonabelian Coxeter groups. Moreover, if an infinite Coxeter group $W_\Gamma$ of rank $n\geq 3$ attains the numerical upper bound in~\eqref{eq:bound2}, then $n=3$ and 
$W_\Gamma$ is isomorphic to a triangle group of the form $\Delta(2p,2q,2r)$ with $p,q,r \in \N$.  
\end{theorem} 

Our main technical result, which goes into the proof of \cref{thm:CardConjClass}, is given below. It provides an explicit criterion for spherical standard parabolic subgroups with nontrivial centers to be conjugate, using odd-adjacency relations (cf. \cref{sec:OddGraphs}) and their Coxeter diagrams. 
A crucial ingredient for the proof is a combinatorial version by Krammer, stated as \Cref{thm:Krammer-Deodhar}, of a theorem essentially due to Deodhar, to be found in \cite[Section 3]{Krammer_2009}. Krammer based his formulation of the statement on Deodhar's result in the case of irreducible subgroups in \cite[Proposition 5.5]{Deodhar_1982}. 
In the finite case this has been proved earlier by Howlett~ \cite{Howlett_1980}. 
Our result reads as follows:

\begin{theorem}[Conjugating parabolics with nontrivial centers] %\yuri{A better name is welcome!}
    \label{thm:CliquesConjugateProd}
    Let $\Gamma=(S,E)$ be a Coxeter diagram with $S = \lbrace s_i \vert \  i \in I \rbrace$ and write $W=W_\Gamma$ for the associated Coxeter group. For $J, K \subseteq I$ let $W_J$ and $W_K$ be the corresponding standard parabolic subgroups, and assume they are spherical with longest elements being central elements. Then the subgroup $W_J$ is conjugate to $W_K$ if and only if both following conditions hold: 
    \begin{enumerate}
        \item there exist $\ell\in\mathbb{N}$ and $L, M ,X \subseteq I$ such that $W_J= W_L\times W_X$ and $W_K=W_M\times W_X$ with $W_L\cong W_M\cong (\mathbb{Z}/2\mathbb{Z})^\ell$, and  
        \item there exists a finite sequence of standard parabolic  subgroups $W_1, \ldots, W_m$ such that $W_1=W_L$, $W_m=W_M$, and $W_i$ is odd-adjacent to $W_{i+1}$ for $i=1,\ldots, m-1$. 
    \end{enumerate}
\end{theorem}

As opposed to Richardson's concept of $W$-equivalence (see below), the conditions stated in our \cref{thm:CliquesConjugateProd} can be interpreted as a more hands-on, combinatorial description of conjugacy of involutions.

We also provide exact formulae on the number $\ccm$ of conjugacy classes of nontrivial elements of  order $m$ in  free and direct products of groups.

\begin{theorem}[Free and direct products]
\label{thm:FreeDirectProducts}
Let $G, H$ be groups and $m\geq 2$.
\begin{enumerate}
\item If $\ccm(G)$ and $\ccm(H)$ are finite, then the number $\ccm$ is additive on free products, that is,
\[
\ccm(G\ast H)=\ccm(G) + \ccm(H).\]

\item If for every $(k,l)\in\mathbb{N}\times\mathbb{N}$ with $\lcm(k,l)=m$, the values $\cck(G)$ and $\ccl(H)$ are finite, then the number of conjugacy classes of elements of order $m$ of a direct product $G\times H$ is also finite and given by

\[
\ccm(G\times H)=\underset{\lcm(k,l)=m}{\sum_{(k,l)\in\N\times\N}} \cck(G)\cdot \ccl(H).
\]
In particular, for a prime number $p$ we have
\[
\ccp(G\times H)=\ccp(G)+ \ccp(H)+\ccp(G)\cdot \ccp(H).
\]
\end{enumerate}
\end{theorem}

We use these results to explicitly compute $\cc$ in multiple nontrivial cases. Most prominently, we give formulae for $\cc(W_\Gamma)$ in the spherical and affine families $\tA$, $\atA$, $\tB=\tC$ and $\atC$, besides explicitly computing $\cc(W_\Gamma)$ for all triangle groups, and all affine irreducible groups of ranks up to eleven. We also give an alternative, geometric way to compute $\cc$ for right-angled Coxeter groups that bypasses our tools. For all such results, see \Cref{sec:specialCases,sec:app:tables-affine}. 

%%%%%%%%%%%%%%%%%%%%%%%%%%%%%%%%%%%%%%%%
\subsection*{Comments on the history of the problem}

We are not the first to describe conjugacy classes in Coxeter groups and to discuss (refinements of) the conjugacy problem. Progress has sometimes been made for certain classes of elements inside a Coxeter group or for certain classes of Coxeter groups. We would like to highlight the following. 

As pointed out, conjugacy classes of involutions were already studied by Deodhar \cite{Deodhar_1982}, Springer \cite{Springer82} and Richardson \cite{Richardson_1982} in the early 1980s.
Building upon Deodhar's work, Springer and Richardson showed independently that every involution in a Coxeter group is in fact conjugate to a unique involution $c_J$ contained in a finite parabolic subgroup satisfying a certain representation-theoretic condition. An immediate consequence of this is seen in \cite[Theorem~A(a)]{Richardson_1982} and \cref{lem:LongestEltsRepresentInvolutions} below in \cref{sec:preliminaries}. The proof of \cref{lem:LongestEltsRepresentInvolutions} bypasses Richardson's tools, making use of Springer's work \cite{Springer82} instead. 
Moreover, Richardson shows in that same paper that 
two such involutions $c_J$ and $c_L$ are conjugate if and only if $J$ and $L$ are $W$-equivalent. 
Subsets $J,L \subseteq I$ are \emph{$W$-equivalent} if there exists $w\in W_\Gamma$ such that $wS_Jw^{-1}=S_L$. 
While this result gives an explicit description of the conjugacy classes of involutions, it takes some time to work out when $S_I$ and $S_J$ 
fullfil the required representation-theoretic condition and 
are $W$-equivalent.  

Carter gave a complete classification  of conjugacy classes in (finite) Weyl groups in terms of graphs in his 1972 paper~\cite{Carter}, building on work of Frobenius \cite{Frobenius}, 
Schur \cite{Schur}, Young \cite{YoungIV, YoungV}, and Specht \cite{Specht}. 

Krammer \cite{Krammer_2009} was first to prove that the classical conjugacy problem is decidable for all Coxeter groups. This led to the question whether there exist minimal length representatives of every conjugacy class and a sequence of operations transforming any element into such a minimal representative in order to decide upon conjugacy -- much like it is possible to decide the word problem by means of braid relations and elementary cancellations of subwords of the form $s_is_i$. 
Marquis \cite{Marquis} provides such a description of conjugate elements in arbitrary Coxeter groups in terms of cyclic shifts. This question can already be found in Cohen's book~\cite{Cohen} and was discussed, for example, by He and Nie \cite{HeNie} in the affine case and by Geck and Pfeiffer \cite{GeckPfeiffer} in the finite case. 

More recently Mili\'cevi\'c, Schwer and Thomas  provide a closed, geometric formula for all conjugacy classes of elements in affine Coxeter groups \cite{MST4} and more generally for  arbitrary split subgroups of the full isometry group of Euclidean space \cite{MST5}. Their formula is in terms of the move-set and fix-set of the linear (i.e. spherical) part of a given element (when written as a product of a translation and an element in the corresponding spherical Coxeter group). 

%%%%%%%%%%%%%%%%%%%%%%%%%%%%%%%%%%%%%%%%
\subsection*{Open problems and future work}

We would like to highlight some natural questions and open problems that remain unsolved in this paper. 

a) Our results, especially the formula in \Cref{thm:CardConjClass}, provide an algorithm to compute the number of conjugacy classes of involutions in a Coxeter group from a given Coxeter diagram. Investigate the complexity of this algorithm and compare it with the algorithm of Deodhar--Howlett described by Richardson \cite{Richardson_1982} and with algorithms provided by the solution to the conjugacy problem in \cite{Krammer_2009}.   

b) Recall from \cref{thm:obereSchranke} that the `universal' upper bound for $\cc(W_\Gamma)$ can only be attained by infinite groups of rank three. That is, the bound $\cc(W_\Gamma) \leq 2^n - 2$ is not optimal if $|V(\Gamma)| \geq 4$ and $W_\Gamma$ is infinite. In turn, lowering the bound down by one to $2^n-3$, one can construct reducible examples in rank four with $\cc(W_\Gamma)=13=2^4-3$; take, for instance, $\Gamma = \atC_2 \sqcup \tA_1$. This leaves open the following question: Does there exist a function $f: \N_{\geq 4} \to \N$ such that $\cc(W) \leq f(n)$ for all infinite irreducible Coxeter groups $W$ of rank $n \geq 4$ and $f(n) = \cc(W_0)$ for some infinite irreducible Coxeter group $W_0$ of rank at least four?

c) Coxeter groups may have torsion elements of any order $m$. Our paper, however, deals with involutions, i.e., the case $m=2$, except for \cref{thm:FreeDirectProducts}. Is it possible to define a similar family of graphs which allows to compute the number of conjugacy classes of elements of a given order $m$ by examining the combinatorial structure of the defining Coxeter diagram? 

d) It is well known that a Coxeter group may admit more than one Coxeter system. While the number of connected components of higher rank odd graphs $\Gamma^k$ depends a priori on the chosen Coxeter diagram for a given group $W$, the number $\cc(W)$ of conjugacy classes of involutions is, by definition, a group-theoretic invariant of $W$. Moreover, \cite[Proposition~2.2]{MoellerVarghese2023} implies that also $|\pi_0(\Gamma_{odd})|=|\pi_0(\Gamma^1)|$ does not depend on $\Gamma$. 
This is interesting in light of what is known about the isomophism problem for Coxeter groups \cite{CoxeterGalaxy}. 
For one, $\cc(W_\Gamma)$ can be used to distinguish nonisomorphic Coxeter groups. On the other hand it would be interesting to see how the $k$-odd graphs behave with respect to moves between Coxeter diagrams. 
Coxeter groups with different Coxeter diagrams are usually obtained by means of combinatorial moves on their diagrams, such as twists \cite{BernhardExample, BMMN} or blow-ups \cite{MuehlherrSurvey}. We broadly ask: how do the $k$-odd graphs of a Coxeter diagram $\Gamma$ change under known `moves' that do not alter the isomorphism type of the underlying Coxeter group?

%%%%%%%%%%%%%%%%%%%%%%%%%%%%%%%%%%%%%%%%
\subsection*{Structure of the paper}

\Cref{sec:preliminaries} sets notation and collects some useful lemmata from the literature. In \Cref{sec:OddGraphs} we introduce higher rank odd graphs, naturally generalizing  odd graphs. In the same section we provide some examples and means to (partially) compute such graphs under specific conditions. \Cref{sec:ConjParabolics} is devoted to the main proofs, in which \Cref{sec:parabolic-technical} concerns itself with technical results about conjugation between parabolic subgroups with nontrivial centers. In \Cref{sec:mainproofs} we prove our main results, namely \Cref{thm:CliquesConjugateProd,thm:CardConjClass,thm:obereSchranke,thm:FreeDirectProducts}, along with corollaries and first examples computing $\cc$. Later in \Cref{sec:specialCases} we explicitly compute and discuss $\cc$ in many subclasses of Coxeter groups, mostly applying \Cref{thm:CardConjClass} and the techniques from \cref{sec:OddGraphs} to obtain our results. Specifically, triangle groups are dealt with in \Cref{sec:triangle}, RACGs are discussed (with an alternative, geometric argument) in \Cref{sec:RACS}, and the final \Cref{sec:finite-affine-Coxeter} concerns some spherical and affine Coxeter groups, with additional explicit values for $\cc$ summarized in the tables from \Cref{sec:app:tables-affine}. 

\subsection*{Acknowledgements} 
AM is fully, and YSR and PS were partially, supported by the DFG, German Research Foundation, grant no. 314838170, GRK 2297 \emph{MathCoRe}. OV is supported by DFG grant VA~1397/2-2. YSR and PS gratefully acknowledge the hospitality of the Fakult\"at f\"ur Mathematik of the OVGU Magdeburg for their stay there in March 2024. The authors thank Paula Macedo Lins de Araujo for helpful comments.

%%%%%%%%%%%%%%%%%%%%%%%%%%%%%%%%%%%%%%%%%
\section{Preliminaries}\label{sec:preliminaries}

This section is used to collect some necessary background and to fix notation. 

A group $W$ is a \emph{Coxeter group} (of finite rank) if it admits an explicit presentation of the following form: 
There exists a subset $S\subset W$ such that $S= \lbrace s_i \mid i\in I\rbrace$ for a finite index set $I = \lbrace 1, \dots, n\rbrace$ and  
\[
W\cong\langle S \mid (s_is_j)^{m_{i,j}} \ \forall i,j\in I\rangle,
\] 
where $m_{i,j}=1$ if $i=j$ and $m_{i,j}=m_{j,i}\in \N_{\geq 2} \cup\{\infty\}$ for $i\neq j$. 
We call the pair $(W,S)$ a \emph{Coxeter system}. The cardinality $n=|I|=|S|$ is called the \emph{rank} of the Coxeter group $W$ (or of the Coxeter system $(W,S)$). Note that the rank is not a group-theoretic invariant of $W$ as it depends on a choice of Coxeter system presenting $W$; see, for instance, \cite{CoxeterGalaxy}. 
One can show that the order of the product $s_is_j$ of two standard generators is exactly $m_{i,j}$ and hence all generators are involutions \cite{Bourbaki_2002}. 
The set of all \emph{reflections} in $W$ with respect to $S$ is then given by the set of conjugates $\{wsw^{-1} \mid w\in W, s\in S\}$. Coxeter systems are usually encoded in certain graphs (or diagrams), detailed below. 

Throughout this text, graphs $\Lambda = (V,E)$ are simplicial and given with a specified vertex set $V=V(\Lambda)$ and edge set $E=E(\Lambda)$. For a vertex subset $X \subseteq V(\Lambda)$, recall that the \emph{full} (or \emph{induced}) \emph{subgraph} spanned by $X$ is the subgraph of $\Lambda$ whose vertex set is $X$ and that contains all edges of $\Lambda$ whose endpoints lie in $X$. For $v\in V$ the \emph{link} of $v$ is defined as $\lk(v)=\left\{w\in V\mid \left\{v,w\right\}\in E\right\}$. A vertex $v\in V$ is called \emph{isolated} if $\lk(v)=\leer$. 
Our graphs are interchangeably called diagrams, and they are usually partially or fully edge-labeled. We tacitly assume `graph (or diagram) isomorphism' to mean isomorphism of edge-labeled graphs. Following standard notation, we let $\pi_0(\Sigma)$ denote the \emph{set of connected components} of a simplicial complex or graph $\Sigma$.

We call the undirected labeled graph $\Gamma:=\Gamma_W = (V,E)$ with vertex set $V(\Gamma)=S$ and edge set $E(\Gamma) \ni \lbrace s_i,s_j \rbrace \iff m_{i,j} \geq 3$, where edges $\lbrace s_i,s_j\rbrace$ with $m_{i,j}\geq 4$ are labeled by $m_{i,j}$ and otherwise unlabeled, the \emph{Coxeter--Dynkin diagram} or simply \emph{Coxeter diagram}. 

Subsets of the index set $I$ are usually denoted by capital letters $J, K, L, M, X$. A \emph{standard parabolic subgroup} of $W$ is a group of the form $W_J := \spans{S_J}$ where $J \subseteq I$ and $S_J := \lbrace s_j \in S \vert \ j \in J \rbrace$. Note that $(W_J, S_J)$ is a Coxeter system for $W_J$ \cite[Chapter IV, Theorem~2]{Bourbaki_2002}. 
A \emph{parabolic subgroup} of $W$ is any subgroup conjugate to a standard parabolic. The trivial group $\{1\}$ and $W$ itself are considered (trivial) parabolic subgroups, corresponding to the (trivial) subsets $\varnothing$ and $I$ of $I$, respectively. 
 For a subset $K\subset W$, the \emph{parabolic closure} of $K$, denoted $\mathrm{Pc}(K)$, is the smallest (with respect to set-theoretic containment) parabolic subgroup of $W$ containing $K$; cf. \cite{Qi2007}.  
 
Given a Coxeter diagram $\Gamma$ we often denote the corresponding Coxeter group $W$ by $W_\Gamma$, which is obtained via a presentation as in the beginning of this section. Though we remind the reader that, if $W = W_\Gamma$ is a Coxeter group associated to the Coxeter diagram $\Gamma$, there might exist another Coxeter diagram $\Gamma'$ such that $W \cong W_{\Gamma'}$ as abstract groups; cf. \cite{BernhardExample,BMMN,CoxeterGalaxy}. Observe that all full subgraphs of a Coxeter diagram induced by a vertex subset $S_J = \lbrace s_j \in S \vert \ j \in J\rbrace \subseteq V(\Gamma) =S$ are also Coxeter diagrams $\Gamma_{W_J}$ of the standard parabolic subgroup $W_J$; cf. \cite[Chapter IV, Theorem~2]{Bourbaki_2002}. To ease notation we write $\Gamma_J$ instead of $\Gamma_{W_J}$. 

A Coxeter diagram $\Gamma$ (or system) and its underlying Coxeter group $W_\Gamma$ is called \emph{reducible} if the graph $\Gamma$ decomposes of a disjoint union of at least two nonempty full subgraphs $\Gamma_1$ and $\Gamma_2$. (In this case, $W_\Gamma$ decomposes as a direct product $W_{\Gamma_1} \times W_{\Gamma_2}$.) Otherwise it is called \emph{irreducible}. Also, $\Gamma$ is called: \emph{spherical}, if $W_\Gamma$ is a finite group; \emph{odd}, if all of its edges are unlabeled or labeled with elements of $2\N+1$; \emph{even}, if all of its edge-labels are contained in $2\N \cup \{\infty\}$. 

A nontrivial element $w \in W_\Gamma$ is an \emph{involution} if the order of $w$ is $2$. 
The \emph{rank of an involution} $w$ is defined as the rank of the {parabolic closure} of $w$. Note that the rank of an involution thus depends on the underlying Coxeter diagram $\Gamma$.

We need a few remarks about longest elements. Recall that the \emph{longest element} in a Coxeter group $W$, usually denoted by $w_0$, is the element with maximal word length with respect to the given Coxeter generating set. It exists only in finite Coxeter groups (and does not otherwise), in which case it is always an involution; see, e.g., \cite[Lemma~4.6.1]{Davis2008}. Sometimes the longest element is a reflection, i.e., it is conjugate to a Coxeter generator. As it turns out, the longest element can also be a central element. 
For example, in $\tA_1\sqcup \tA_1$ with nodes $s_1$ and $s_2$ the longest element in $W_{\tA_1\sqcup \tA_1}$ is $s_1s_2$, which is central and has rank $2$. The longest element in $W_{\tA_3}$, where $V(\tA_3) = \{s_1, s_2, s_3\}$, is the word $s_1s_2s_3s_2s_1$, which is a rank $1$ noncentral reflection.

The following observation is well-known: 

\begin{lemma}[Center of a Coxeter group]
    \label{IrreducibleWithCenter}
    Let $W$ be an irreducible Coxeter group. The center of $W$ is nontrivial if and only if $W\cong W_\Gamma$ with $\Gamma$ isomorphic to one of the following (spherical) Coxeter diagrams: $\tA_1$, $\tB_n$ (with $n \geq 2$), $\tD_{2n}$ (with $n \geq 2$), $\tE_7$, $\tE_8$, $\tG$, $\tF$, $\tH_3$, $\tH_4$ or $\tI(2m)$ (with $m\geq 2$); cf. \Cref{fig:Tabelle}. Moreover, in such cases the center of $W_\Gamma$ is precisely the subgroup $\spans{w_0}$ of order two, where $w_0$ is the longest element of $W_\Gamma$.
\end{lemma}

\begin{proof}
In the infinite case cf. \cite[Theorem~D.2.10]{Davis2008}, \cite[Proposition~2.73]{AbramenkoBrown} or \cite{Qi2009}. The first claim in the spherical case follows from Coxeter's classification of finite irreducible Coxeter groups; see, for instance, \cite[Preliminaries, 1.9--1.12]{Richardson_1982} or \cite[Remark~13.1.8]{Davis2008}. The second claim has been long known for Weyl groups (\cite{Carter}), and can be readily checked for the remaining types $\tH$ and $\mathtt{I}$; for two elementary proofs, we refer the reader to \cite[Corollary~1.91]{AbramenkoBrown}. 
\end{proof}

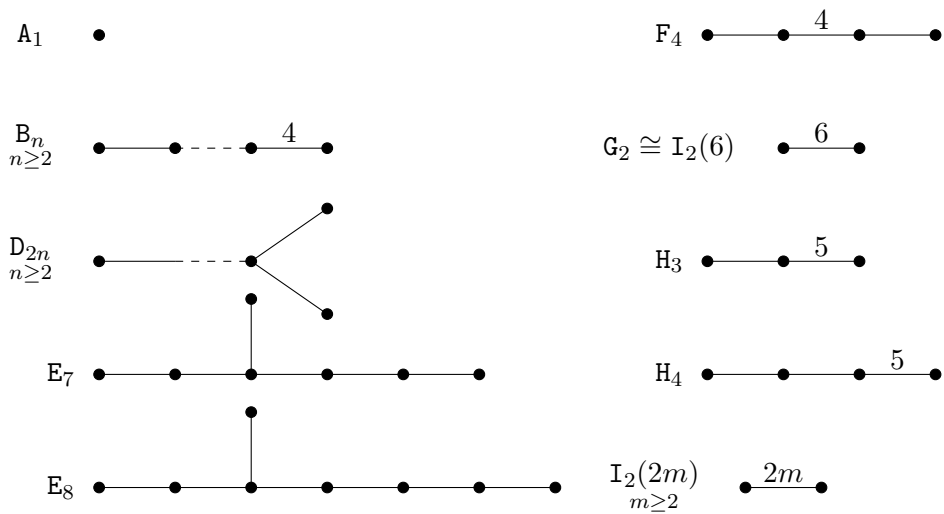
\begin{figure}[htb!]
    \begin{center}
        \captionsetup{justification=centering}
        \begin{tikzpicture}
            \draw[fill=black]  (0,0) circle (2pt);
            \node at (-0.9,0) {$\tA_{1}$};
            \draw[fill=black]  (0,-1.5) circle (2pt);
            \draw[fill=black]  (1,-1.5) circle (2pt);
            \draw[fill=black]  (2,-1.5) circle (2pt);
            \draw (0,-1.5)--(1,-1.5);
            \draw[dashed] (1,-1.5)--(2,-1.5);
            \draw[fill=black] (3,-1.5) circle (2pt);
            \draw (2,-1.5)--(3,-1.5);
            \node at (2.5, -1.3) {$4$};
            \node at (-0.9,-1.5) {$\underset{n\geq2}{\tB_{n}}$};
            
            \draw[fill=black]  (0,-3) circle (2pt);
            %\draw[fill=black]  (1,-3) circle (2pt);
            \draw[fill=black]  (2,-3) circle (2pt);
            \draw (0,-3)--(1,-3);
            \draw[dashed] (1,-3)--(2,-3);
            \draw[fill=black] (3,-2.3) circle (2pt);
            \draw[fill=black] (3,-3.7) circle (2pt);
            \draw (2,-3)--(3, -2.3);
            \draw (2,-3)--(3, -3.7);
            \node at (-0.9,-3) {$\underset{n\geq 2}{\tD_{2n}}$};
    
            \draw (0,-4.5)--(5,-4.5);
            \draw (2,-3.5)--(2,-4.5);
            \draw[fill=black] (0,-4.5) circle (2pt); 
            \draw[fill=black] (1,-4.5) circle (2pt);
            \draw[fill=black] (2,-3.5) circle (2pt);
            \draw[fill=black] (3,-4.5) circle (2pt);
            \draw[fill=black] (4,-4.5) circle (2pt);
            \draw[fill=black] (5,-4.5) circle (2pt);
            \draw[fill=black] (2,-4.5) circle (2pt);
            \node at (-0.5,-4.5) {$\tE_7$};
    
            \draw (0,-6)--(6,-6);
            \draw (2,-5)--(2,-6);
            \draw[fill=black] (0,-6) circle (2pt); 
            \draw[fill=black] (1,-6) circle (2pt);
            \draw[fill=black] (2,-6) circle (2pt);
            \draw[fill=black] (3,-6) circle (2pt);
            \draw[fill=black] (4,-6) circle (2pt);
            \draw[fill=black] (5,-6) circle (2pt);
            \draw[fill=black] (6,-6) circle (2pt);
            \draw[fill=black] (2,-5) circle (2pt);
            \node at (-0.5,-6) {$\tE_8$};

            \draw[fill=black] (8,0) circle (2pt);
            \draw[fill=black] (9,0) circle (2pt);
            \draw[fill=black] (10,0) circle (2pt);
            \draw[fill=black] (11,0) circle (2pt);
            \draw (8,0)--(11,0);
            \node at (9.5, 0.2) {$4$};
            \node at (7.5,0) {$\tF$};
       
            \draw[fill=black] (9,-1.5) circle (2pt);
            \draw[fill=black] (10,-1.5) circle (2pt);
            \draw (9,-1.5)--(10,-1.5);
            \node at (9.5,-1.3)  {$6$};
            \node at (7.5,-1.5) {$\tG \cong \tI(6)$};
    
            \draw[fill=black] (8,-3) circle (2pt);
            \draw[fill=black] (9,-3) circle (2pt);
            \draw[fill=black] (10,-3) circle (2pt);
            \draw (8,-3)--(10,-3);
            \node at (9.5, -2.8){$5$};
            \node at (7.5,-3) {$\tH_3$};
    
            \draw[fill=black] (8,-4.5) circle (2pt);
            \draw[fill=black] (9,-4.5) circle (2pt);
            \draw[fill=black] (10,-4.5) circle (2pt);
            \draw[fill=black] (11,-4.5) circle (2pt);
            \draw (8,-4.5)--(11, -4.5);
            \node at (10.5, -4.3){$5$};
            \node at (7.5,-4.5) {$\tH_4$};
      
            \draw[fill=black] (8.5,-6) circle (2pt);
            \draw[fill=black] (9.5,-6) circle (2pt);
            \draw (8.5,-6)--(9.5,-6);
            \node at (9, -5.8){$2m$};
            \node at (7.3,-6) {$\underset{m \geq 2}{\tI(2m)}$};
                    
            \end{tikzpicture}
        \caption{The Coxeter diagrams yielding irreducible Coxeter groups with nontrivial centres.}
        \label{fig:Tabelle}
        \end{center}
    \end{figure}

Besides determining the center of (spherical) Coxeter groups, longest elements can also be used to represent conjugacy classes of involutions.

\begin{lemma}[Involutions and longest elements] \label{lem:LongestEltsRepresentInvolutions}
    Let $(W,S)$ be a Coxeter system of finite rank $|I|$, with $S = \{s_i \mid i \in I\}$, and let $c \in W$ be an involution. Then there exists a subset $J \subseteq I$ of smallest possible cardinality $|J| \geq 1$ such that $W_J$ is spherical, the longest element $w_0$ of $W_J$ is central, and $c$ is conjugate to $w_0$.
\end{lemma}

\begin{proof}
    Given the involution $c$, consider its parabolic closure $\mathrm{Pc}(c) \subseteq W$. Since $c$ has finite order, a known theorem of Tits (cf. \cite[Corollary~D.2.9]{Davis2008}) together with minimality of $\mathrm{Pc}(c)$ imply that the parabolic $\mathrm{Pc}(c)$ is in fact spherical. Moreover, it is conjugate to some standard parabolic $W_K$ with $|K|$ minimal; see, e.g., \cite[Theorem~3.4]{Qi2007}. Thus $c$ is conjugate to some involution in the finite Coxeter group $W_K$. In case $W_K$ is a Weyl group, i.e., with irreducible components of classical types $\tA_n$ to $\tG$, apply \cite[Proposition~3]{Springer82} to conclude that such an involution is conjugate to the longest element of a (unique, standard) parabolic subgroup $W_J \subseteq W_K$ (with $J \subseteq K$) whose longest element is central. In case types $\tH_n$ or $\tI(n)$ are involved, one can directly check that the conclusion of Springer's proposition still holds, whence the lemma.
\end{proof}

We remark that the content of \cref{lem:LongestEltsRepresentInvolutions} is also part of Richardson's main results in the paper \cite{Richardson_1982}, namely \cite[Theorem~A(a)]{Richardson_1982}. While our proof is shorter and somewhat more elementary, we make use of Springer's work \cite{Springer82} done independently of Richardson's.

\begin{lemma}[Conjugated longest elements] \label{lem:ConjLongest}
    Suppose $W_J$ and $W_K$ are standard parabolic subgroups such that their irredicuble components are of types in \Cref{fig:Tabelle}. Let $c_J$ and $c_K$ be the longest elements of $W_J$ and $W_K$, respectively. Then $\mathrm{Pc}(c_J) = W_J$, $\mathrm{Pc}(c_K) = W_K$, and $c_J$ is conjugate to $c_K$ if and only if $W_J$ and $W_K$ are conjugate.
\end{lemma}

\begin{proof}
    By \cref{IrreducibleWithCenter}, $c_J$ is central. Since a reduced expression for $c_J$ involves all letters of $S_J$ and $c_J$ commutes with all such generators, 
it follows that $\mathrm{Pc}(c_J) = W_J$. 
(Similarly for $c_K$.) Now assuming $c_J = w c_K w^{-1}$, we obtain $W_J=\mathrm{Pc}(c_J) = \mathrm{Pc}(w c_K w^{-1}) = w\mathrm{Pc}(c_K)w^{-1} = w W_K w^{-1}$. The converse follows from \cref{lem:LongestEltsRepresentInvolutions}. 
\end{proof}

In particular, \Cref{IrreducibleWithCenter}, \Cref{lem:LongestEltsRepresentInvolutions} and \Cref{lem:ConjLongest} show that for a given Coxeter group $W_\Gamma$ the number of conjugacy classes of involutions $\cc(W_\Gamma)$ is equal to the number of conjugacy classes of standard parabolic subgroups of $W_\Gamma$ such that their irreducible components are of types in \Cref{fig:Tabelle}. 

%%%%%%%%%%%%%%%%%%%%%%%%%%%%%%%%%%%%%%%%%
\section{Odd-adjacency and higher rank odd graphs} 
\label{sec:OddGraphs}

In this section we generalize the concept of odd graphs and determine the number of conjugacy classes of involutions by counting connected components of these graphs. 

\begin{definition}[Adjacent subgraphs]
    \label{def:adj-graphs}
    Let $\Gamma$ be a Coxeter diagram for the Coxeter system $(W,S)$. Two full subgraphs $\Gamma_{J}$ and $\Gamma_{K}$ are \emph{adjacent} if the diagrams $\Gamma_{J}$ and $\Gamma_{K}$ are isomorphic and $V(\Gamma_{J})=S_J$ differs from $V(\Gamma_{K})= S_K$ only by one vertex. 
\end{definition}

For later use we remark the following observation on adjacent Coxeter diagrams:

\begin{lemma}
\label{lem:onevertex}
    Let $\Gamma$ be a Coxeter diagram and let $\Gamma_{J}$ and $\Gamma_{K}$ be two adjacent full subgraphs, and let $s_j \in V(\Gamma_{J}) \setminus V(\Gamma_{K})$ and $s_k \in V(\Gamma_{K}) \setminus V(\Gamma_{J})$ be the unique vertices for which the two graphs differ from one another. 
If $\Gamma_{J}$ (equivalently, $\Gamma_{K}$) is of type $\tB_n$, $\tD_{2n}$, $\tE_7$, $\tE_8$, $\tG$, $\tF$, $\tH_3$, $\tH_4$ or $\tI(2m)$ where $n, m \geq 2$, then the standard parabolic subgroup with generating set $V(\Gamma_{J}) \cup V(\Gamma_{K}) = S_{J \cup K }$ is infinite. 
\end{lemma}
\begin{proof}
    Note that 
the vertex $s_j$ (resp. $s_k$) is always connected to some vertex in $V(\Gamma_{J}) \cap V(\Gamma_{K})$. 
Thus the claim follows from the characterization of finite irreducible Coxeter groups.
    (For the list of finite irreducible Coxeter groups, see \cite[Chapter~VI, Theorem~1]{Bourbaki_2002}.)
\end{proof}

We now aim to generalize odd-adjacency for higher rank parabolic subgroups in order to define higher rank analogs of the odd-graphs appearing in \cite{BMMN}. 

\begin{definition}[Odd-adjacent standard parabolics]
    \label{def:odd-adjacent}
    Let $\Gamma$ be a Coxeter diagram. 
    Two full subgraphs $\Gamma_{J}$ and $\Gamma_{K}$  are called \emph{odd-adjacent} if  there exist $\ell\geq 1$ and $S_L, S_M, S_X \subseteq S$ such that 
    \begin{enumerate}
        \item $W_J= W_L\times W_X$, $W_K=W_M\times W_X$;
        \item $W_L\cong W_M\cong (\mathbb{Z}/2\mathbb{Z})^\ell$;
        \item $S_L$ differs from $S_M$ only by one generator; and
        \item there is an edge $\left\{s_l,s_m\right\}$ that is odd or unlabeled, where $s_l\in S_L\setminus S_M$ and $s_m\in S_M\setminus S_L$.
    \end{enumerate}
    In this case, we also say that $W_J$ and $W_K$ are \emph{odd-adjacent}. If $J=\{s\}$, $K=\{t\}$ are singletons and $\Gamma_J$ and $\Gamma_K$ are odd-adjacent, we also call the vertices $s,t\in V(\Gamma)$ odd-adjacent in $\Gamma$, in which case we call the edge $\{s,t\} \in E(\Gamma)$ an \emph{odd edge}.
\end{definition} 

\begin{remark}[Odd-adjacency implies adjacency]
    Note that, in the set-up of \cref{def:odd-adjacent}, the subgraphs $\Gamma_{J}$ and $\Gamma_{K}$ are adjacent in the sense of \cref{def:adj-graphs}. This follows from the decomposition in item~(1) and the condition in item~(3) since $V(\Gamma_J) = L \cup X$ and $V(\Gamma_K) = M \cup X$. 
\end{remark}

\begin{example} 
To illustrate odd-adjacency, let us consider the affine Coxeter group $W_{\atE_6}$, where $\atE_6$ is the Coxeter diagram in \Cref{fig:affineE6}. Then the parabolic subgroup $W_{\left\{1,3,7\right\}}$ is odd-adjacent to the parabolic subgoup $W_{\left\{1,4,7\right\}}$.
\end{example}
\begin{figure}[htb!]
	\begin{center}
	\captionsetup{justification=centering}
		\begin{tikzpicture}[scale=1, transform shape]
            \draw[fill=black] (0,0) circle (2pt);
            \node at (0,-0.3) {$s_1$};
            \draw[fill=black] (1,0) circle (2pt);
            \node at (1,-0.3) {$s_2$};
            \draw[fill=black] (2,0) circle (2pt);
            \node at (2,-0.3) {$s_3$};
            \draw[fill=black] (3,0) circle (2pt);
            \node at (3,-0.3) {$s_4$};
            \draw[fill=black] (4,0) circle (2pt);
            \node at (4,-0.3) {$s_5$};
            \draw[fill=black] (2,1) circle (2pt);
            \node at (2.3,1) {$s_6$};
            \draw[fill=black] (2,2) circle (2pt);
            \node at (2.3, 2) {$s_7$};
            \draw (0,0)--(4,0);
            \draw (2,0)--(2,2);
\end{tikzpicture}
	\end{center}
\caption{Type $\atE_6$.}
 \label{fig:affineE6}
\end{figure}

We now introduce the class of $k$-odd graphs. Although adjacency is defined for all parabolics, in the next definition only parabolics with nontrivial center play a role. 
    
\begin{definition}[Higher rank odd graphs]
\label{def:odd-graphs}
    Let $\Gamma$ be a Coxeter diagram with vertex set $S$ indexed by $I$.
    For $k \geq 1$, the \emph{$k$-odd graph} $\Gamma^k = (V({\Gamma^k}), E({\Gamma^k}))$ is defined as follows:
    \begin{itemize}
        \item The set of vertices $V(\Gamma^k)$ is the set of rank-$k$ standard parabolic subgroups $W_J$ for $J\subseteq I$ with the property that the subgraph of $\Gamma$ induced by $S_J$ splits into connected components all of which are of types listed in \Cref{fig:Tabelle}; 
        \item A pair of vertices $W_J, W_K$ spans an edge in $\Gamma^k$ if and only if $W_J$ and $W_K$ are odd-adjacent.
    \end{itemize}
\end{definition}

Note that $\Gamma^1=\Gamma_{odd}$ with labels removed, and $\Gamma^k = \leer$ if 
$k > |V(\Gamma)|$. Moreover, it is known that $|\pi_0(\Gamma^1)|$ is independent of the chosen Coxeter system $(W,S)$, see \cite[Proposition 2.2]{MoellerVarghese2023}.
Below we include some first examples of higher odd graphs. 

\begin{example}[Universal diagram] \label{ex:oddgraphsuniversal} 
Let $\Gamma$ be a complete graph on $n$ vertices, all of whose edge labels are $\infty$ as shown for $n=4$ in \Cref{fig:UniversalCoxeterGraph}. The corresponding Coxeter group $W_\Gamma$ is called a \emph{universal} Coxeter group as it maps onto any Coxeter group of rank $n$. In this case one readily checks that $\Gamma^1$ has $n$ isolated vertices, and all other $\Gamma^k$ with $k > 1$ are empty since no diagrams with $\infty$-labels occur in \Cref{fig:Tabelle}.
\end{example}

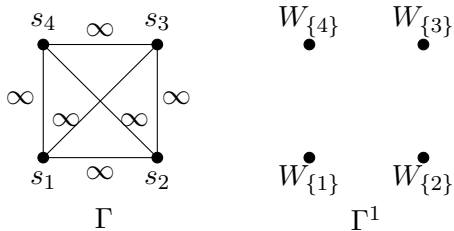
\begin{figure}[h]
	\begin{center}
	\captionsetup{justification=centering}
		\begin{tikzpicture}
			\draw[fill=black]  (0,0) circle (2pt);
			\draw[fill=black]  (1.5,0) circle (2pt);
            \draw[fill=black]  (1.5,1.5) circle (2pt);
			\draw[fill=black]  (0,1.5) circle (2pt);
            \node at (0,-0.3) {$s_1$};
            \node at (1.5,-0.3) {$s_2$};
            \node at (0,1.8) {$s_4$};
            \node at (1.5,1.8) {$s_3$};
            \draw (0,0)--(1.5,0);
            \draw (0,0)--(1.5, 1.5);
            \draw (0,0)--(0,1.5);
            \draw (1.5,0)--(1.5, 1.5);
            \draw (1.5,0)--(0,1.5);
            \draw (0,1.5)--(1.5, 1.5);
            \node at (0.8,-0.8) {$\Gamma$};
            
            \node at (-0.3,0.8) {$\infty$};
            \node at (0.75,-0.2) {$\infty$};
            \node at (1.75,0.8) {$\infty$};
            \node at (0.75,1.7) {$\infty$};
            \node at (0.3,0.5) {$\infty$};
            \node at (1.2,0.5) {$\infty$};

            \draw[fill=black]  (3.5,0) circle (2pt);
			\draw[fill=black]  (5,0) circle (2pt);
            \draw[fill=black]  (5,1.5) circle (2pt);
			\draw[fill=black]  (3.5,1.5) circle (2pt);
            \node at (3.5,-0.3) {$W_{\left\{1\right\}}$};
            \node at (5,-0.3) {$W_{\left\{2\right\}}$};
            \node at (3.5,1.8) {$W_{\left\{4\right\}}$};
            \node at (5,1.8) {$W_{\left\{3\right\}}$};
            \node at (4.25,-0.8) {$\Gamma^1$};

\end{tikzpicture}
    \end{center}
    \caption{Universal diagram $\Gamma$ with $4$ vertices and its $1$-odd graph $\Gamma^1$.}
    \label{fig:UniversalCoxeterGraph}
\end{figure}   

\begin{example}[Disconnected diagram]
\label{ex:oddgraphsabelian}
   Suppose the Coxeter diagram consists of $n$ isolated vertices as shown for $n=4$ in \Cref{fig:DisconnectedDiagram}. That is, all generators of $W_\Gamma$ commute and hence $W_\Gamma \cong (\Z/2\Z)^n$, an elementary abelian Coxeter group. Then any $k$-element subset of $V(\Gamma)$ yields a subdiagram consisting of $k$ copies of the $\tA_1$ diagram. In particular, any such vertex subset gives rise to a vertex of $\Gamma^k$. As no odd or unlabeled edges occur in $\Gamma$, no pair of standard parabolics in $\Gamma$ can be odd-adjacent. Hence $\Gamma^k$ consists of exactly $\binom{n}{k}$ isolated vertices for every $k \in \left\{1,\ldots,n\right\}$.
\end{example}

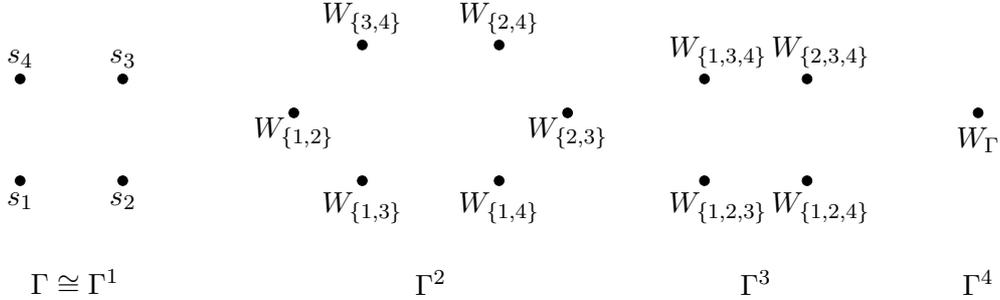
\begin{figure}[h]
	\begin{center}
	\captionsetup{justification=centering}
		\begin{tikzpicture}[scale=0.9]
			\draw[fill=black]  (0,0) circle (2pt);
			\draw[fill=black]  (1.5,0) circle (2pt);
            \draw[fill=black]  (1.5,1.5) circle (2pt);
			\draw[fill=black]  (0,1.5) circle (2pt);
            \node at (0,-0.3) {$s_1$};
            \node at (1.5,-0.3) {$s_2$};
            \node at (0,1.8) {$s_4$};
            \node at (1.5,1.8) {$s_3$};
            \node at (0.8,-1.5) {$\Gamma\cong\Gamma^1$};

            \draw[fill=black]  (4,1) circle (2pt);
            \node at (4,0.7) {$W_{\{1,2\}}$};
			\draw[fill=black]  (5,0) circle (2pt);
            \node at (5,-0.4) {$W_{\{1,3\}}$};
            \draw[fill=black]  (7,0) circle (2pt);
            \node at (7,-0.4) {$W_{\{1,4\}}$};
			\draw[fill=black]  (8,1) circle (2pt);
            \node at (8,0.7) {$W_{\{2,3\}}$};
            \draw[fill=black]  (7,2) circle (2pt);
            \node at (7,2.4) {$W_{\{2,4\}}$};
			\draw[fill=black]  (5,2) circle (2pt);
            \node at (5,2.4) {$W_{\{3,4\}}$};
            \node at (6,-1.5) {$\Gamma^2$};

            \draw[fill=black]  (10,0) circle (2pt);
            \node at (10.2,-0.4) {$W_{\{1,2,3\}}$};
			\draw[fill=black]  (11.5,0) circle (2pt);
            \node at (11.7,-0.4) {$W_{\{1,2,4\}}$};
            \draw[fill=black]  (11.5,1.5) circle (2pt);
            \node at (10.2,1.9) {$W_{\{1,3,4\}}$};
			\draw[fill=black]  (10,1.5) circle (2pt);
            \node at (11.7,1.9) {$W_{\{2,3,4\}}$};
            \node at (10.75,-1.5) {$\Gamma^3$};

            \draw[fill=black]  (14,1) circle (2pt);
            \node at (14,0.6) {$W_\Gamma$};
            \node at (14,-1.5) {$\Gamma^4$};
            
            \end{tikzpicture}
    \end{center}
    \caption{Disconnected diagram $\Gamma$ with $4$ vertices and its $k$-odd graphs.}
    \label{fig:DisconnectedDiagram}
\end{figure} 

\begin{example}[Type $\tA$]
\label{ex:A4}
    Consider the finite Coxeter group of type $\tA_4$ with Coxeter diagram as shown in Figure~\ref{fig:graphsA4}. Since all edges in $\tA_4$ are unlabeled, $\Gamma^1$ is isomorphic to $\tA_4$. The second odd graph, $\Gamma^2$, is also displayed in Figure~\ref{fig:graphsA4}. The remaining odd graphs $\Gamma^k$ are empty for $k\geq 3$. 
\end{example}

\begin{figure}[h]
	\begin{center}
	\captionsetup{justification=centering}
		\begin{tikzpicture}
			\draw[fill=black]  (0,0) circle (2pt);
			\draw[fill=black]  (1,0) circle (2pt);
            \draw[fill=black]  (2,0) circle (2pt);
            \draw[fill=black] (3,0) circle (2pt);
            \node at (0,-0.3) {$s_1$};
            \node at (1,-0.3) {$s_2$};
            \node at (2,-0.3) {$s_3$};
            \node at (3,-0.3) {$s_4$};
            \draw (0,0)--(3,0);
			\node at (1.5,-1) {$\tA_{4}\cong\Gamma^1$};

            \draw[fill=black]  (5,0) circle (2pt);
			\draw[fill=black]  (7,0) circle (2pt);
            \draw[fill=black]  (9,0) circle (2pt);           
            \node at (5,-0.3) {$W_{\{1,3\}}$};
            \node at (7,-0.3) {$W_{\{1,4\}}$};
            \node at (9,-0.3) {$W_{\{2,4\}}$};
            \draw (5, 0)--(9, 0);
			\node at (7,-1) {$\Gamma^2$};
        \end{tikzpicture}
    \end{center}
    \caption{Pictured are the two nonempty odd graphs in type $\tA_4$. 
}
    \label{fig:graphsA4}
\end{figure}
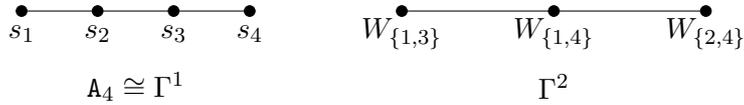

\begin{remark}
    Although we have introduced higher rank odd graphs $\Gamma^k$ as graphs themselves, derived from a Coxeter diagram $\Gamma$, it turns out that $(V(\Gamma),V(\Gamma^k))$ can be interpreted a $k$-uniform hypergraph. In this set-up, connected components in each $\Gamma^k$ correspond to certain subgraphs of $(k-1)$-line graphs. We refer the reader to \cite{BrettoBook} for more on hypergraph theory, and to \cite{BermondHeydemannSotteau} for an introduction to $s$-line graphs.
\end{remark}

In the sequel we offer some observations on how to quickly derive edges of $\Gamma^k$ under certain conditions. The first criterion is general, while the following ones work for diagrams containing certain `line subdiagrams'. 
We first fix our notation on neighbourhoods in graphs.

\begin{definition}[Graph neighbourhoods] \label{def:neighbourhoods}
Let $\Gamma = (S,E)$ be a graph. 
We define the \emph{neighbourhood} $N_{\Gamma}(K)$ of a subset of vertices $K \subseteq S$ in $\Gamma$ as the subgraph of $\Gamma$ induced by 
the vertex set $\left\{s_i\in S\setminus K\mid \text{ there exists }s_k\in K\text{ such that }\left\{s_i, s_k\right\}\in E\right\}$.

In particular, for $K = \lbrace s_i, s_j \rbrace$, the neighbourhood $N_\Gamma(K)$ is the full subgraph of $\Gamma$ with the vertex set $(\lk(s_i)\cup \lk(s_j)) \setminus \left\{s_i, s_j\right\}$ and $\Gamma\setminus N_\Gamma(K)$ is the full subgraph with the vertex set $(S\setminus \lk(s_i)\cup \lk(s_j))\cup\left\{s_i,s_j\right\}$.
\end{definition}

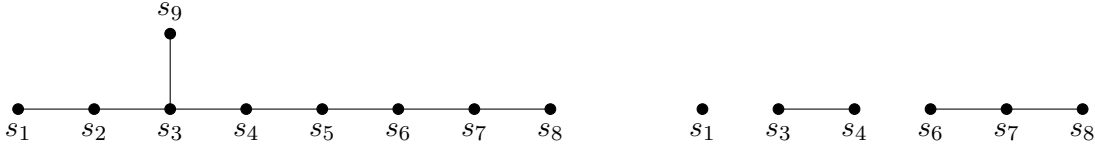
\begin{figure}[h]
	\begin{center}
	\captionsetup{justification=centering}
		\begin{tikzpicture}

            \draw (-1,-4.5)--(6,-4.5);
            \draw (1,-3.5)--(1,-4.5);
            \draw[fill=black] (-1,-4.5) circle (2pt);
            \draw[fill=black] (0,-4.5) circle (2pt); 
            \draw[fill=black] (1,-4.5) circle (2pt);
            %\draw[fill=black] (2,-3.5) circle (2pt);
            \draw[fill=black] (3,-4.5) circle (2pt);
            \draw[fill=black] (4,-4.5) circle (2pt);
            \draw[fill=black] (5,-4.5) circle (2pt);
            \draw[fill=black] (6,-4.5) circle (2pt);
            \draw[fill=black] (2,-4.5) circle (2pt);
            \draw[fill=black] (1,-3.5) circle (2pt);
            \node at (-1,-4.8) {$s_1$};
            \node at (0,-4.8) {$s_2$};
            \node at (1,-4.8) {$s_3$};
            \node at (2,-4.8) {$s_4$};
            \node at (3,-4.8) {$s_5$};
            \node at (4,-4.8) {$s_6$};
            \node at (5,-4.8) {$s_7$};
            \node at (6,-4.8) {$s_8$};
            \node at (1,-3.2) {$s_9$};	
            %\node at (2,-5.8) {$\atE_8$};

            \draw[fill=black] (8,-4.5) circle (2pt);
            \node at (8,-4.8) {$s_1$};
            \draw[fill=black] (9,-4.5) circle (2pt);
            \draw[fill=black] (10,-4.5) circle (2pt);
            \draw[fill=black] (11,-4.5) circle (2pt);
            \draw[fill=black] (12,-4.5) circle (2pt);
            \draw[fill=black] (13,-4.5) circle (2pt);
            \draw (9,-4.5)--(10, -4.5);
            \node at (9,-4.8) {$s_3$};
            \node at (10,-4.8) {$s_4$};
            \node at (11,-4.8) {$s_6$};
            \node at (12,-4.8) {$s_7$};
            \node at (13,-4.8) {$s_8$};
            \draw (11,-4.5)--(13, -4.5);
		\end{tikzpicture}
	\caption{Type $\atE_8$ and $\atE_8\setminus N_{\atE_8}(s_3,s_4)$.}
	\end{center}
\end{figure}

For a given odd-edge $\left\{s_i, s_j\right\}\in E(\Gamma)$ the next lemma describes precisely which standard parabolic subgroups whose irreducible components are of types in \Cref{fig:Tabelle} are odd-adjacent  via this edge.

\begin{lemma}[Edges in $\Gamma^m$]
    \label{lemma:edges_Gamma_n}
    Let $\Gamma = (S,E)$ be a Coxeter diagram
    and $\lbrace s_i,s_j \rbrace \in E$ be an odd-edge, and let 
    $\til{\Gamma} = \Gamma \setminus N_\Gamma (s_i,s_j)$.  
    If $m\geq 2$,   
    then $\lbrace \langle S_K, s_i \rangle , \langle S_K, s_j \rangle \rbrace \in E(\Gamma^m)$ if and only if $S_K\subseteq V(\til{\Gamma})\setminus\lbrace s_i,s_j\rbrace$ with cardinality $|S_K|=m-1$ and such that the connected components of the subgraph induced by $S_K$ are of types listed in \Cref{fig:Tabelle}. 
  In particular, $\langle S_K, s_i \rangle$ and $\langle S_K, s_j \rangle$ are standard parabolic subgroups of $W_\Gamma$ of rank $m$ with irreducible components of types in \Cref{fig:Tabelle}.
\end{lemma}
\begin{proof}
    Let first  $\lbrace \langle S_K, s_i \rangle , \langle S_K, s_j \rangle \rbrace \in E(\Gamma^m)$. 
    Then the vertices $\langle S_K, s_i \rangle$ and $\langle S_K, s_j \rangle$ of $\Gamma^m$ are standard parabolic subgroups of rank $m$ with irreducible components of types in \Cref{fig:Tabelle}. 
    In particular, $\left\vert S_K \cup \lbrace s_i \rbrace\right\vert = m$ implies $\left\vert S_K \right\vert = m-1$. 
    Since the two vertices form an edge in $\Gamma^m$, the parabolic subgroups $ \langle S_K, s_i \rangle$ and $\langle S_K, s_j \rangle$ are odd-adjacent.  And it follows from \cref{def:odd-adjacent} that $\langle S_K, s_i \rangle = W_L \times W_X$, and that $\langle S_K, s_j \rangle = W_M \times W_X$, 
    for some index sets $L, X \subseteq K \cup \lbrace i \rbrace$, and  $M,X \subseteq K \cup \lbrace j \rbrace$. In addition, $W_L \cong W_M \cong (\mathbb{Z}/2\mathbb{Z})^\ell$ for some $\ell \geq 1$.
    Moreover,
    $\left\vert L \setminus M \right\vert = \left\vert M \setminus L \right\vert =1$ with $\lbrace i \rbrace = L \setminus M$, $\lbrace j \rbrace = M \setminus L$ and the edge $\lbrace s_i, s_j \rbrace \in E(\Gamma)$ is odd.
    Note that the standard parabolic subgroups  $\langle s_i \rangle$, $\langle s_j\rangle$ commute with  $W_L, W_M$ and $W_X$.
    This implies that the intersection of $V(N_\Gamma(s_i, s_j))$ with $S_L\cup S_M\cup S_X$ is empty.
    Hence  $S_K \subseteq V(\til{\Gamma})\setminus\lbrace s_i, s_j\rbrace$. 
    
    Conversely, let $S_K\subseteq V(\til{\Gamma})\setminus\lbrace s_i, s_j\rbrace$, of cardinality $|S_K|=m-1$ such that the connected components of the subgraph induced by $S_K$ are of types in \Cref{fig:Tabelle}.
    Then the standard parabolic subgroups $\langle S_K, s_i \rangle$ and $\langle S_K, s_j \rangle$ are of rank $m$. Further, the irreducible components of $\langle S_K, s_i \rangle$ and of  $\langle S_K, s_j \rangle$ are of types listed in \Cref{fig:Tabelle} since $s_i$ and $s_j$ are not adjacent to any vertex in $S_K$ and therefore $\langle S_K, s_i\rangle\cong W_K\times \langle s_i\rangle$ resp. $\langle S_K, s_j\rangle\cong W_K\times\langle s_j\rangle$.
    By assumption the edge $\lbrace s_i,s_j \rbrace \in E(\Gamma)$ is odd, which implies $\lbrace \langle S_K, s_i \rangle , \langle S_K, s_j \rangle \rbrace \in E(\Gamma^m)$.
\end{proof}

The next tools concern line graphs that avoid subdiagrams of type $\tH$. These graphs are particularly useful as they show up as subdiagrams in most classical families of spherical and affine Coxeter diagrams.
We start with the \emph{Klapperschlangen}-Lemma. 

\begin{lemma}[Edge paths in $\Gamma^k$ from odd sublines] \label{lem:foldinglines}
Let $\Gamma$ be a Coxeter diagram and let $W_J,W_L \in V(\Gamma^k)$ be distinct vertices in the $k$-odd graph $\Gamma^k$ for some $k \geq 1$. Suppose there exists a full subgraph $\Lambda \subseteq \Gamma$ satisfying all of the following conditions: 
\begin{enumerate}
    \item The full subgraphs $\Gamma_J$ and $\Gamma_L$ corresponding to the given standard parabolics $W_J$ and $W_L$ are both contained in $\Lambda$;
    \item $\Lambda$ is a \emph{path graph} (i.e., a tree with two leaves) such that all of its edges are odd or unlabelled;
    \item $\Lambda$ contains no subdiagram of type $\tH_3$. 
\end{enumerate}
Then there exists an edge path connecting $W_J$ to $W_L$ in the $k$-odd graph $\Gamma^k$. 
\end{lemma}

\begin{proof}
The intuitive idea of the proof is to `play' with the path graph $\Lambda$ as a `Klapperschlange' toy\footnote{From the German `folding snake', a literal translation of the common name for a rattle snake; `Klapperschlangen' toys (or `Zauberklapperschlangen') have wooden bits of same size that can be folded.}, folding it step-by-step to connect the given parabolics. 

As $\Lambda$ is a path graph by Condition~(2), we may linearly order its vertices as follows: 
\[V(\Lambda) = \{ {a_1}, \ldots, {a_m} \},\] 
where $a_1$ and $a_m$ are the endpoints (leaves) of $\Lambda$ and $a_{i+1}$ is the obvious successor of $a_i$ in this linear ordering. Note that any edges in $\Gamma$ connecting any two $a_i, a_j \in V(\Lambda)$ must already be contained in $\Lambda$ since $\Lambda$ is assumed to be a full subgraph of $\Gamma$. By definition, $\Gamma_J$ and $\Gamma_L$ are also full subgraphs of $\Gamma$, which in turn are contained, by Condition~(1), in the full subgraph $\Lambda$. By assumption $\Gamma_J \neq \Gamma_L$ and hence $|V(\Lambda)| = m > k = |V(\Gamma_J)| = |V(\Gamma_J)|$. By Conditions~(2) and~(3) and after inspecting \Cref{fig:Tabelle}, we conclude that the irreducible components of $\Gamma_J$ and $\Gamma_L$ are all of type $\tA_1$. That is, $\Gamma_J$ and $\Gamma_L$ consist, respectively, of $k$ vertices which are pairwise not connected by any edges from $\Gamma$. 
We may thus write 
\begin{align*}
V(\Gamma_J) &= \{a_{j_1}, \ldots, a_{j_k} \mid j_1 < j_2 < \ldots < j_k \text{ and } |j_f - j_h | \geq 2 \text{ for } f,h \in \{1,\ldots,k\}\}, \\
V(\Gamma_L) &= \{a_{\ell_1}, \ldots, a_{\ell_k} \mid \ell_1 < \ell_2 < \ldots < \ell_k \text{ and } |\ell_f - \ell_h | \geq 2 \text{ for } f,h \in \{1,\ldots,k\}\}.
\end{align*} 

Now suppose $\Gamma_J$ and $\Gamma_L$ only differ by one vertex, say $a_{j_\phee} \in V(\Gamma_J) \setminus V(\Gamma_L)$ and $a_{\ell_\lambda} \in V(\Gamma_L) \setminus V(\Gamma_J)$ for some $\phee,\lambda \in \{1,\ldots,k\}$ with $\phee\neq \lambda$. Without loss of generality, assume $j_{\phee} < \ell_{\lambda} = j_{\phee} + \eta$ with $\eta \geq 1$. In case $\eta = 1$, it is clear that $W_J$ and $W_L$ are odd-adjacent and thus connected in $\Gamma^k$ by an edge. So assume $\eta > 1$, in which case $|V(\Lambda)| = m > k + 1$. Looking at the intersection $V(\Gamma_J) \cap V(\Gamma_L)$ we see that, for every other vertex $a_{j_h} \in V(\Gamma_J) \setminus \{a_{j_\phee}\}$ and $a_{\ell_f} \in V(\Gamma_L) \setminus \{a_{\ell_\lambda}\}$, it holds 
\[|j_h-\ell_\lambda| \geq 2 \leq |\ell_f - j_\phee|.\]
This means that none of the vertices $a_{j_{\phee}+1}$, $a_{j_{\phee}+2}$, $\ldots$, $a_{j_{\phee}+\eta-1} = a_{\ell_\lambda-1}$ of the path graph $\Lambda$ lying between $a_{j_\phee}$ and $a_{\ell_\lambda}$ is contained in $V(\Gamma_J) \cup V(\Gamma_L)$. Moreover, since every other $a_{j_h}$ (resp. $a_{\ell_f}$) of $V(\Gamma_J) \cap V(\Gamma_L)$ lies before $a_{j_\phee}$ or after $a_{\ell_\lambda}$, it still holds true that 
\[|j_h-(\ell_\lambda+i)| \geq 2 \leq |\ell_f - (j_\phee+i)| \quad \text{ for all } \quad i \in \{1,\ldots,\eta-1\}\]
and all $h,f \in \{1,\ldots,k\}$ with $h\neq \phee$ and $f \neq \lambda$. Hence, defining for each $i \in \{1,\ldots,\eta-1\}$ the vertex set 
$V_i := \{a_{j_\phee+i}\} \cup V(\Gamma_J) \setminus \{a_{j_\phee}\} \subseteq \Lambda$,
the parabolic subgroup $W_i := \spans{V_i}$ corresponds to a vertex of $\Gamma^k$ and is odd-adjacent to $W_{i+1}$, additionally with $W_1$ being odd-adjacent to $W_J$ and $W_{\eta-1}$ being odd-adjacent to $W_L$. Hence $W_J$, $W_1$, $\dots$, $W_{\eta-1}$, $W_L$ yields an edge path in $\Gamma^k$ connecting $W_J$ to $W_L$. 

In the general case, an inductive argument finishes the proof. Indeed, if $\Gamma_J$ and $\Gamma_L$ differ by more than one vertex (hence $|V(\Lambda)| = m \geq k+2$), we can iterate the procedure above to first construct an edge path in $\Gamma^k$ from $W_J$ to another vertex $W_{J'} \in V(\Gamma^k)$ with $V(\Gamma_{J'}) \subseteq V(\Lambda)$ such that $\Gamma_{J'}$ and $\Gamma_L$ differ by one vertex less than $\Gamma_J$ and $\Gamma_L$. The lemma follows. 
\end{proof}

The (sub)diagrams $\Lambda$ appearing in \Cref{lem:foldinglines} are depicted in \Cref{fig:path}. A typical example of such a graph is the spherical Coxeter diagram of type $\tA_n$. The arguments used in the proof of \Cref{lem:foldinglines} also point out to the following observation. 

\begin{figure}[htb]
	\begin{center}
	\captionsetup{justification=centering}
		\begin{tikzpicture}[scale=1, transform shape]
            \draw[fill=black] (0,-4.5) circle (2pt);
            \draw[fill=black] (1,-4.5) circle (2pt);
            \draw (0,-4.5)--(1, -4.5);
            \draw[fill=black] (2, -4.5) circle (2pt);
            \draw (1, -4.5)--(2,-4.5);
            \draw[dashed] (2,-4.5)--(3,-4.5);
            \draw[fill=black] (3,-4.5) circle (2pt);
            \draw[fill=black] (4.25,-4.5) circle (2pt);
            \draw (3,-4.5)--(4.25,-4.5);
            \node at (0.5, -4.3) {$m_{1,2}$};
            \node at (1.5, -4.3) {$m_{2,3}$};
            \node at (3.6, -4.3) {$m_{n-1,n}$};
            \node at (0, -4.8) {$s_1$};
            \node at (1, -4.8) {$s_2$};
            \node at (2, -4.8) {$s_3$};
            \node at (3.05, -4.8) {$s_{n-1}$};
            \node at (4.25, -4.8) {$s_n$};

\end{tikzpicture}
	\caption{Odd-labeled line graph without $\tH$-subdiagrams; if $m_{i,i+1}=3$, the label is understood as omitted.}
     \label{fig:path}
	\end{center}
\end{figure}
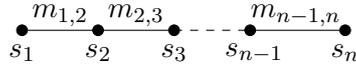  

\begin{proposition}[Vertices and connectivity in odd lines]
\label{lemma:VerticesInTypeA_path}
Let $\Lambda$ be a Coxeter line diagram as in \Cref{fig:path}, with edges unlabeled or odd, with $n = |V(\Lambda)| \geq 2$, and containing no subdiagram isomorphic to $\tH_3$. Then $\Lambda^1 \cong \Lambda$ (as unlabeled, simplicial graphs) and the vertex set of $\Lambda^k$ for $k \geq 2$ is given as follows:
    \begin{align*}
        V(\Lambda^k) = \left\lbrace \langle s_{j_1}, \dots, s_{j_k}\rangle \vert \ j_1 < j_2 < \dots < j_k,\  \left|  j_f - j_h\right| \geq 2,\ f,h \in \lbrace 1, \dots, k \rbrace \right\rbrace, 
    \end{align*}
in case $2 \leq k \leq \lceil \frac{n}{2} \rceil$; 
otherwise $V(\Gamma^k) = \leer$. Moreover, for $2 \leq k < n$, the odd graph $\Lambda^k$ consists of a single vertex if and only if $n=2\ell+1$ and $k=\ell+1$; otherwise, if $\Lambda^k$ has at least two vertices, then it is always connected. 
\end{proposition}

\begin{proof}
    By definition the only (relevant) subdiagrams of types listed in \Cref{fig:Tabelle} that are contained in such a $\Lambda$ are of type $\tA_1$. This yields the description of the vertices of $\Lambda^k$. In particular, pigeonholing vertices of $\Lambda$ to obtain potential vertices of $\Lambda^k$, one readily checks that $\Lambda^k$ is empty for $k > \lceil \frac{n}{2} \rceil$. 
By the same principle, the only vertex of $\Lambda^{\ell+1}$ in case $n=2\ell+1$ corresponds precisely to the parabolic subgroup spanned by all vertices of $\Lambda$ indexed by an odd number. The remaining claim about connectivity of $\Lambda^k$ is a direct consequence of \Cref{lem:foldinglines}.
\end{proof}

After analyzing odd lines, one quickly deduces similar properties for higher odd graphs of `odd circles', depicted in \Cref{fig:Circle}. A typical example of such a circle is the affine Coxeter diagram of type $\atA_n$. In such a diagram we call $s_{i+1}$ with $i+1 \mod{n+1}$ the \emph{successor} of $s_i$, and similarly $s_{i-1}$ with $i-1 \mod {n+1}$ is the \emph{predecessor} of $s_i$.

    \begin{figure}[htb]
	\begin{center}
	\captionsetup{justification=centering}
		\begin{tikzpicture}[scale=1, transform shape]
            \draw[fill=black] (0,-1.5) circle (2pt);
            \draw[fill=black] (1,-1.5) circle (2pt);
            \draw (0,-1.5)--(1,-1.5);
            \draw[fill=black] (2,-1.5) circle (2pt);
            \draw[fill=black] (4.3,-1.5) circle (2pt);
            \draw[dashed] (1,-1.5)--(2, -1.5);
            \draw (2,-1.5)--(4.3,-1.5);
            \draw[fill=black] (1.5, -0.8) circle (2pt);
            \draw (0,-1.5)--(1.5, -0.8);
            \draw (4.3,-1.5)--(1.5, -0.8);
            \node at (-0.2, -1.85) {$s_0$};
            \node at (4.3, -1.85) {$s_{n-1}$};
            \node at (1.5, -0.5) {$s_n$};

            \node at (0.5, -1) {$m_{0,n}$};
            \node at (2.8, -0.95) {$m_{n,n-1}$};
            \node at (0.5, -1.75) {$m_{0,1}$};
            \node at (3, -1.75) {$m_{n-2,n-1}$};
    \end{tikzpicture}
	\caption{Odd-labeled circle without $\tH$-subdiagrams; if an edge label equals three, the label is understood as omitted.}
     \label{fig:Circle}
	\end{center}
\end{figure}
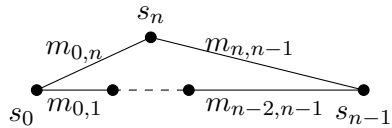

\begin{lemma}[Vertices in odd circles] 
\label{lemma:VerticesInTypeA}
     Fix  $n \geq 2$ and let $\Gamma$ be a Coxeter diagram as shown in \Cref{fig:Circle}. That is, $\Gamma$ is a circle on $n+1$ vertices $S= \lbrace s_i \;\vert \ i \in \lbrace 0, \dots, n \rbrace \rbrace$ whose edges are unlabeled or odd. Suppose further that $\Gamma$ contains no subdiagram of type $\tH_3$. 
    Then $\Gamma^1$ is isomorphic to $\Gamma$ (as unlabeled, simplicial graphs) and the vertices $V({\Gamma^k})$ of $\Gamma^k$, with $2 \leq k \leq n$, are of the form 
    \begin{align*}
        V({\Gamma^k}) = \left\{ \langle s_{j_1}, \dots, s_{j_k}\rangle \mid j_1 < j_2 < \dots < j_k, \, \left|  j_f - j_h\right| \geq 2 \, \mathrm{mod} \, {n+1}, f,h \in \lbrace 0, \dots, n \rbrace \right\}.
    \end{align*}
    In particular, $\Gamma^k$ is empty if and only if $k > \lceil \frac{n}{2}\rceil$. In case $k \leq \lceil \frac{n}{2} \rceil$, a vertex $\langle s_{j_1}, \dots, s_{j_k}\rangle \in V({\Gamma^k})$ is isolated if and only if 
    \begin{enumerate}
        \item $j_{m+1}= j_m + 2 \ \forall m \in \lbrace 1, \dots, k \rbrace$, and
        \item $j_k + 2 = j_1 \mod{n+1}$.
    \end{enumerate}
\end{lemma}

\begin{proof} 
The description of $V(\Gamma^k)$ is obtained similarly to that in \Cref{lemma:VerticesInTypeA_path}. 
    Now, since every vertex $s_i$ of $\Gamma$ is odd-adjacent to $s_{i-1}$ and to $s_{i+1} \mod{n+1}$, we can always find in $\Gamma^k$ a vertex odd-adjacent to $\langle s_{j_1}, \dots, s_{j_k}\rangle$ by exchanging one of the $s_{j_m}$ by $s_{j_m +1}$ or $s_{j_m -1} \mod{n+1}$, except possibly when 
    \begin{enumerate}
        \item $j_{m+1}= j_m + 2$ for all $m \in \lbrace 1, \dots, k\rbrace$ --- in this case, every predecessor and successor of a generator $s_{j_m}$ of the corresponding parabolic in $V(\Gamma^k)$ is odd-adjacent in $\Gamma$ to another generator of the same parabolic. This contradicts the fact that the given parabolic has irreducible components from \Cref{fig:Tabelle}; and 
        \item $j_{k}+2 = j_1 \mod{n+1}$ --- here, also the successor of $s_{j_k}$ is adjacent in $\Gamma^k$ to the predecessor of $s_{j_1}$.
    \end{enumerate}
    Thus no rank-$k$ parabolic that is odd-adjacent to the given vertex of $\Gamma^k$ can be found.
\end{proof}

The previous results lay down some foundations to implement a practical algorithm to compute the higher rank odd graphs of Coxeter diagrams. In \Cref{sec:specialCases} we shall make free use of such results to compute $\cc$ for Coxeter groups in many nontrivial examples, after establishing \Cref{thm:CardConjClass} in the next section.

%%%%%%%%%%%%%%%%%%%%%%%%%%%%%%%%%%%%%%%%%%%%%%%%%%%%%%%%%%%%%%
\section{Conjugate parabolic subgroups}
\label{sec:ConjParabolics}

The main goal of this section is to prove \Cref{thm:CliquesConjugateProd,thm:CardConjClass}, to be done in \Cref{sec:mainproofs}. To this end we need a series of auxiliary results describing conjugation between parabolic subgroups, which we prove below in \Cref{sec:parabolic-technical}. 

\subsection{Technical results concerning conjugation of parabolics}
\label{sec:parabolic-technical}
We start the discussion with a combinatorial reformulation of a result by Krammer~\cite[Section 3]{Krammer_2009} and Deodhar~\cite[Proposition 5.5]{Deodhar_1982}.

\begin{theorem}[Combinatorial Deodhar--Krammer Theorem]
    \label{thm:Krammer-Deodhar}
    Let $W_\Gamma$ be a Coxeter group with generating set $S = \lbrace s_i \vert \ i \in I \rbrace$ and $W_J$, $W_K$, $J, K \subseteq I$, $J\neq K$ be  spherical standard parabolic subgroups. 
    The subgroups $W_J$ and $W_K$ are conjugate if and only if
    \begin{enumerate}
    \item there exist a finite sequence of subdiagrams $\Gamma_1, \ldots, \Gamma_m\subseteq\Gamma$ such that $V(\Gamma_1)=J$, $V(\Gamma_m)=K$ and $\Gamma_i$ is adjacent to $\Gamma_{i+1}$ for $i=1\ldots, m-1$,
    \item for $i=1,\ldots, m-1$ the subgroup $W_{V(\Gamma_i)\cup V(\Gamma_{i+1})}$ is spherical,  and
    \item for $i=1,\ldots, m-1$ there exists $a_i\in W_{V(\Gamma_i)\cup V(\Gamma_{i+1})}$ such that $a_iW_{\Gamma_i}a_i^{-1}=W_{\Gamma_{i+1}}$. 
    \end{enumerate}
\end{theorem}

Together with our observation on adjacent Coxeter diagrams in \Cref{lem:onevertex} above, we obtain the following characterization of conjugacy for irreducible spherical standard parabolic subgroups with nontrivial center from this: 

\begin{proposition}
    \label{IrreducibleAdjacent}
    Let $W_\Gamma$ be a Coxeter group with generating set $S= \lbrace s_i \vert \  i \in I \rbrace$. Let $J,K \subseteq I$, $J\neq K$ be such that $W_J$ and $W_K$ are irreducible spherical standard parabolic subgroups with nontrivial centers. Then the subgroup $W_J$ is conjugate to $W_K$ if and only if $J=\left\{s_j\right\}$, $ K=\left\{s_k\right\}$, $s_j, s_k\in V(\Gamma)$ and there exists an odd-path in $\Gamma$ from $s_j$ to $s_k$.
\end{proposition}
\begin{proof}
    \Cref{thm:Krammer-Deodhar} shows that if $W_J$ and $W_K$ are conjugate, then there exists a finite sequence of subgraphs $\Gamma_1, \ldots, \Gamma_m$ such that $V(\Gamma_1)=J$, $V(\Gamma_m)=K$, $\Gamma_i$ is adjacent to $\Gamma_{i+1}$, and the standard parabolic subgroups generated by $V(\Gamma_i)\cup V(\Gamma_{i+1})$ are finite. By \cref{IrreducibleWithCenter} and \Cref{lem:onevertex} it follows that $W_J$ and $W_K$ are cyclic of order $2$. 
    Further, it was proven  in \cite[p.~5, Proposition~3]{Bourbaki_2002} that $s_j$ is conjugate to $s_k$ if and only if there exist elements $s_1,\ldots, s_{\ell}\in V(\Gamma)$ such that $s_1=s_j, s_{\ell}=s_k$ and $\ord(s_i s_{i+1})$ is odd for $i=1,\dots, \ell-1$, which is equivalent to saying that there exists an odd path in $\Gamma$ from $s_j$ to $s_k$. For the sake of completeness we include a proof of this fact here. We define $A_{s_i}:=\left\{s\in V(\Gamma)\mid\text{ there exists an odd path from }s_i\text{ to } s\right\}$. It is clear that each $s\in A_{s_i}$ is conjugate to $s_i$. Now we define a map $\varphi\colon V(\Gamma)\to \mathbb{Z}/2\mathbb{Z}$ as follows: $\varphi(s)=0$ if $v\in A_{s_i}$ and $\varphi(s)=1$ if $v\notin A_{s_i}$. An easy calculation shows that this map extends to $\Phi\colon W_\Gamma\to\mathbb{Z}/2\mathbb{Z}$. Now, if $s_j\in V(\Gamma)$ is conjugate to $s_i$, then $\Phi(s_j)=0$ and since $\Phi$ is the extension of $\varphi$, we have $\Phi(s_j)=\varphi(s_j)=0$ which shows that $s_j\in A_{s_i}$.
\end{proof}

For conjugated pairs of (not necessarily reduced) standard parabolics the following holds:

\begin{lemma}(Conjugacy of irreducible components)
    \label{Products}
    Let $W_\Gamma$ be a Coxeter group. Let $W_J$ and $W_K$ be 
standard parabolic subgroups of $W_\Gamma$. If $W_J$ is conjugate to $W_K$, then the irreducible components of $W_J$ are conjugate to the irreducible components of $W_K$.
\end{lemma} 
\begin{proof}
By \cite[Proposition 4.5.10]{Davis2008} the subgroup $W_J$ is conjugate to $W_K$ if and only if there exists $w\in W_\Gamma$ such that $wS_Jw^{-1}=S_K$. 
Further, the conjugation by $w$ induces an isomorphism of graphs $\Gamma_J$ and $\Gamma_K$. Precisely, for $s_i, s_j\in S_J$ and writing $s_a = w s_i w^{-1}$ and $s_b = w s_j w^{-1}$, we have $m_{i,j}= m_{a,b}$. 
That is, labels are preserved and the irreducible components of $W_J$ are conjugate to the irreducible components of $W_K$.
\end{proof}
    
Combining \Cref{IrreducibleAdjacent} with \Cref{Products} we obtain a splitting for the conjugated subgroups as follows:
    
\begin{corollary}
\label{ProductsZ2}
    Let $W_\Gamma$ be a Coxeter group and let $W_J$ and $W_K$ 
be distinct spherical standard parabolic subgroups whose longest elements are central. 
    If $W_J$ is conjugate to $W_K$, then there exists $\ell\in\mathbb{N}$ and a spherical standard parabolic subgroup $W_X$ such that  $W_J=(\mathbb{Z}/2\mathbb{Z})^\ell\times W_X$, $W_K=(\mathbb{Z}/2\mathbb{Z})^\ell\times W_X$, and $\mathbb{Z}/2\mathbb{Z}$ is not a direct factor of $W_X$.
\end{corollary}
\begin{proof}
Let $W_J=W_{X_1}\times\ldots\times W_{X_n}$ and $W_K=W_{Y_1}\times\ldots\times W_{Y_m}$ be the direct decompositions into irreducible components, i.e., the subgraphs induced by the $X_i$ (resp. the $Y_j$) are the connected components of the subgraph induced by $J$ (resp. by $K$). Since the longest element in $W_J$ resp. in $W_K$ is central, the subgraphs of $\Gamma$ induced by $X_1, \ldots, X_n$ resp. $Y_1,\ldots, Y_m$ are of types listed in \Cref{fig:Tabelle}. 

If $W_J$ is conjugated to $W_K$, then \Cref{Products} shows that $n=m$ and there exists a permutation $\varphi\colon\left\{1,\ldots, n\right\}\to\left\{1,\ldots, n\right\}$ such that for $i=1,\ldots, n$ the standard parabolic subroup  $W_{X_i}$ is conjugated to $W_{Y_{\varphi(i)}}$. Let $i\in \{1,\ldots, n\}$. If the subgraph induced by $X_i$ is not of type $\tA_1$, then
by \Cref{IrreducibleAdjacent} we obtain the equality $X_i=Y_{\varphi(i)}$.

By assumption $J\neq K$. Hence, after isolating the $\Z/2\Z$ factors of the direct product decompositions for $W_J$ and $W_K$ above, we obtain some $\ell\in\mathbb{N}$ and a (spherical) standard parabolic subgroup $W_X$ such that  $W_J=(\mathbb{Z}/2\mathbb{Z})^\ell\times W_X$ and $W_K=(\mathbb{Z}/2\mathbb{Z})^\ell\times W_X$ with $\mathbb{Z}/2\mathbb{Z}$ not a direct factor of $W_X$.
\end{proof}

A final ingredient for the proof of \Cref{thm:CliquesConjugateProd} is the following proposition describing conjugacy of parabolics by means of connecting odd-paths. 

\begin{proposition}(Conjugate, odd-connected parabolics)
    \label{prop:CliquesConjugate}
    Let $W_\Gamma$ be a Coxeter group. Let $W_J\cong(\mathbb{Z}/2\mathbb{Z})^\ell$ and $W_K\cong(\mathbb{Z}/2\mathbb{Z})^\ell$ be elementary abelian spherical standard parabolic subgroups of same rank. Then 
    $W_J$ is conjugate to $W_K$ if and only if there exists 
    a finite sequence of standard parabolic  subgroups $W_1, \ldots, W_m$ such that $W_1=W_J$, $W_m=W_K$ and $W_i$ is odd-adjacent to $W_{i+1}$ for $i=1,\ldots, m-1$.
\end{proposition}

\begin{proof}
    Let $W_J \cong W_K \cong (\mathbb{Z}/2\mathbb{Z})^\ell$ be conjugated.  
    By \Cref{thm:Krammer-Deodhar} there exists a finite sequence of standard parabolic subgroups $W_{\Gamma_1}, \dots, W_{\Gamma_m}$ with $V(\Gamma_1)=J, \ V(\Gamma_m)=K$ and $\Gamma_i$ adjacent to $\Gamma_{i+1}$ for all $i = 1, \dots, m-1$. In particular, $\Gamma_i$ and $\Gamma_{i+1}$ differ by a single vertex. 
    Moreover, \Cref{lem:onevertex} together with parts~(2) and~(3) of \Cref{thm:Krammer-Deodhar} implies that the vertices that differ between $\Gamma_i$ and $\Gamma_{i+1}$ must be connected via an odd-labeled edge in $\Gamma$ since $W_{\Gamma_i}$ and $W_{\Gamma_{i+1}}$ are conjugated via an element $a_i \in W_{V(\Gamma_i) \cup V(\Gamma_{i+1})}$. Hence $W_{\Gamma_i}$ is odd-adjacent to $W_{\Gamma_{i+1}}$. 
    
    Conversely, let $W_J =W_1, W_2, \dots, W_m = W_K$ be a finite sequence of odd-adjacent standard parabolic subgroups. 
    Let $\lbrace s_1, \dots, s_e\rbrace$ be a generating set for $W_i$ and $\lbrace s_1, \dots, s_{e-1}, t_e\rbrace$ for $W_{i+1}$ with $i \in \lbrace 1, \dots, m-1\rbrace$.
    Since $W_i$ and $W_{i+1}$ are odd-adjacent, $\ord(s_e t_e)=q$ is odd.
    Define $w_i:= s_e t_e s_e\cdots t_e$ where the length of $w_i$ is $2q-1$.
    Then an easy calculation shows that $w_i W_i w_i^{-1} = W_{i+1}$. Iterating this process, it follows that the subgroups $W_J$ and $W_K$ are conjugated. 
\end{proof}

\subsection{Proofs of the main results} \label{sec:mainproofs}

We are now ready to prove our main technical result, \Cref{thm:CliquesConjugateProd}.

\begin{proof}[Proof of \Cref{thm:CliquesConjugateProd}]
    If the subgroups $W_J$ and $W_K$ fulfill the conditions stated and are conjugate, \Cref{ProductsZ2} shows that there exist index subsets $L,M,X \subseteq I$ and a splitting $W_J = W_L \times W_X$, $W_K=W_M \times W_X$ as described in item~(1). 
    Furthermore, item~(2) follows by applying \Cref{prop:CliquesConjugate} to the abelian factors $W_L$ and $W_M$. 
    Conversely, if $W_J$ and $W_K$ satisfy items~(1) and~(2), the claim follows immediately from \Cref{prop:CliquesConjugate}, again applied to the abelian factors.
\end{proof}

Our main result, \cref{thm:CardConjClass}, is now an easy consequence of \cref{thm:CliquesConjugateProd} and previous lemmata.

\begin{proof}[Proof of \Cref{thm:CardConjClass}]
Given an arbitrary involution $c \in W_\Gamma$, apply \cref{lem:LongestEltsRepresentInvolutions} to find a central longest element in a spherical parabolic that represents its conjugacy class. (Note that this representative might well be a standard reflection.) This association is unique up to conjugation by \cref{lem:ConjLongest}. The set of conjugacy classes of involutions is thus completely described once we know which spherical parabolics of $W_\Gamma$ whose longest elements are central are conjugate to one another. But \cref{thm:CliquesConjugateProd} shows that these conjugacy classes correspond bijectively to connected components of the $k$-odd graphs $\Gamma^k$. The theorem follows, as does \cref{cor:conjClassRefinement}.
\end{proof}

\begin{proof}[Proof of \Cref{cor:unterSchranke}]
Let $\Gamma$ be a Coxeter diagram of rank $n$. Since $\Gamma^1$ is always nonempty, by \Cref{thm:CardConjClass} there is only one conjugacy class of involutions in $W_\Gamma$ if and only of $|\pi_0(\Gamma^1)|=1$ and the vertex set of $\Gamma^k$ is empty for $2\leq k\leq n$. 
By definition of higher rank odd graphs, $V(\Gamma^k)$ is empty for $2\leq k\leq n$ if and only if $\Gamma$ is complete and every edge in $\Gamma$ is odd or $\infty$. This proves~(1).

Further, by \Cref{thm:CardConjClass} every involution in $W_\Gamma$ is a reflection if and only if $V(\Gamma^k)$ is empty for $2\leq k\leq n$ . The proof of part~(1) shows that this is the case if and only if $\Gamma$ is complete and every edge label in $\Gamma$ is odd or $\infty$, whence item~(2).
\end{proof}

Recall from the introduction that for a given Coxeter group $W$ the total number of conjugacy classes of involutions is denoted by $\mathrm{cc}_2(W)$. Having proved our main theorems, we quickly compute $\cc(W)$ for the easy examples from \cref{sec:OddGraphs}. Further (more intricate) examples are discussed in \Cref{sec:specialCases}. 

\begin{example}[Universal Coxeter groups] \label{ex:universalCoxetergroups}
Standard facts from Bass--Serre theory imply that $\mathrm{cc}_2(W_\Gamma)=n$ for for the universal Coxeter group $W_\Gamma \cong \ast_{i=1}^n \Z/2\Z$ of rank $n$. Alternatively from \cref{ex:universalCoxetergroups} and \cref{thm:CardConjClass} we readily obtain $\cc(\displaystyle\ast_{i=1}^n \Z/2\Z) = n$. 
\end{example}

\begin{example}[Elementary abelian groups]
\label{ex:finiteAbelian}
   It is clear that a finite abelian Coxeter group $W = (\Z/2\Z)^n$ of rank $n$ satisfies $\mathrm{cc}_2(W) = 2^n-1 = \vert W \setminus \{1\}\vert$. The reason is that, in this case, conjugacy classes bijectively correspond to group elements and every nontrivial element is an involution. The same count can be obtained from \cref{ex:oddgraphsabelian} and the formula from \Cref{thm:CardConjClass}:  
    \[
    \mathrm{cc}_2\left(\prod_{i=1}^n \Z/2\Z\right) = \sum_{k=1}^{n}\vert\pi_0(\Gamma^k)\vert=\sum_{k=1}^{n}\binom{n}{k}=2^n-1. 
    \] 
\end{example} 

\begin{example}[Type $\tA$]
From \cref{ex:A4} and \cref{thm:CardConjClass} we know that $\cc(W_\Gamma) = 2$ in case $\Gamma$ is the spherical diagram of type $\tA_4$. A general formula for $\cc(W_{\tA_n})$ will be given in \Cref{thm:finiteCoxeter}. 
\end{example}

In what follows we prove Theorem~\ref{thm:obereSchranke}. Before diving into the proof we need the following concept.

\begin{definition}[$\mc{O}$-graphs] \label{def:omegas}
Let $\Gamma$ be a Coxeter diagram. The first $\mc{O}$-graph $\Omega^1$ is defined as $\Gamma^1$ itself and, for $k\geq 2$, we define the graph $\Omega^k$ as follows: the vertex set of $\Omega^k$ is equal to the vertex set of $\Gamma^k$, and two vertices $W_I, W_J$ in $\Omega^k$ are connected if and only if $W_I$ is isomorphic to $W_J$.  
\end{definition}

\begin{proof}[Proof of \Cref{thm:obereSchranke}]
We start with the numerical bounds from Inequality~(2) in the statement. Note that the lower bound $1 \leq \cc(W_\Gamma)$ is immediate, and sharp by \cref{cor:unterSchranke}. We now deduce the upper bounds $2^n-1$ and $2^n-2$. 

There are at most $\binom{n}{i}$ standard rank-$i$ parabolics in any Coxeter group $W_\Gamma$ of rank $n$. Hence the number of connected components in the $k$-odd graphs satisfies
\[
\vert\pi_0(\Gamma^i)\vert \leq \vert V(\Gamma^i)\vert \leq \binom{n}{i},
\]
and thus 
\[ 
\mathrm{cc}_2(W) = \sum_{k=1}^{n}\vert\pi_0(\Gamma^k)\vert \leq  \sum_{k=1}^{n}\binom{n}{k} = 2^{n}-1.
\]
In case $W$ is finite, the inequalities above become equalities for elementary abelian Coxeter groups and the given upper bound for $\mathrm{cc}_2(W_\Gamma)$ is sharp, as seen in \cref{ex:finiteAbelian}. For a nonabelian (spherical) example, see \cref{ex:nonabelianupper} below.

The group $W_\Gamma$ is trivially a (nonproper) parabolic subgroup of itself, of rank $n=\vert V(\Gamma)\vert$. Hence the $n$-odd graph $\Gamma^n$ is nonempty and contains a single vertex if and only if $W_\Gamma$ is a direct product of finite Coxeter groups whose types are listed in \Cref{fig:Tabelle}, in which case $W_\Gamma$ itself is also finite. 
Therefore the last summand for $k=n$ must be $0$ in the infinite case. This implies that
\[ 
\mathrm{cc}_2(W_\Gamma) = \sum_{k=1}^{n}\vert\pi_0(\Gamma^k)\vert = 0+\sum_{k=1}^{n-1}\vert\pi_0(\Gamma^k)\vert \leq  \sum_{k=1}^{n}\binom{n}{k} -1 = 2^{n}-2
\]
is an upper bound in the infinite case. 
Note that the the affine group of type $\atC_2$ % triangle group $W_\Gamma\cong\Delta(2,4,4)$
has the property that $\Gamma^1$ and $\Gamma^2$ are totally disconnected with maximal number of vertices. Hence $\mathrm{cc}_2(W_{\atC_2}) = 3 + 3 = \binom{3}{1}+\binom{3}{2}+0=2^{3}-2$ and the upper bound
 is sharp in the infinite case as well. 

Suppose that $\vert V(\Gamma)\vert = n \geq 3$ and the upper bounds are attained. Then, for all $k$, every rank-$k$ proper standard parabolic subgroup of $W_\Gamma$ gives rise to a vertex of $\Gamma^k$. Moreover, all vertices in $\Gamma^k$ are isolated and there is no odd-adjacent pair of parabolics in $V(\Gamma^k)$. So $W_\Gamma$ has no edge with an odd label (as otherwise there would be odd-adjacent rank-$1$ subgroups) and in particular the rank-$2$ subgroups are finite and no edge is labeled $\infty$. This is saying $W_\Gamma$ is even and $2$-spherical.  

We argue that no infinite Coxeter group $W_\Gamma$ of rank $\vert V(\Gamma)\vert = n > 3$ attains the maximum. Assume, on the contrary, that this is the case. Then all $k$-odd graphs contribute positively to the formula for $\cc(W_\Gamma)$ and we can estimate the number of their connected components as follows: We have already observed that the rank-$2$ subgroups of $W_\Gamma$ are finite; So let us consider the $3$-odd graph. 
We have that $\vert\pi_0(\Gamma^3)\vert = \binom{n}{3} \geq \binom{4}{3} = 4$. 
Exhausting the list in \Cref{fig:Tabelle}, if all rank-$3$ parabolics are of type $\tA_1 \sqcup \tA_1 \sqcup \tA_1$, then all generators of $W_\Gamma$ commute with one another, so $W_\Gamma$ is elementary abelian (and finite). Thus $\Gamma^3$ must have at least one vertex of type different than $\tA_1 \sqcup \tA_1 \sqcup \tA_1$. By the previous paragraph, though, we know that $\Gamma$ is an even diagram. Thus consulting back the list in \Cref{fig:Tabelle}, we see that the only possible rank-$3$ even finite parabolic that can occur (besides $\tA_1 \sqcup \tA_1 \sqcup \tA_1$) is $\tA_1 \sqcup \tI(2r)$. 
Hence the group $W_\Gamma$ must contain at least one parabolic of type $\tA_1 \sqcup \tI(2r)$. 

Further, we claim that there exists no path in the diagram $\Gamma$ connecting the above $\tA_1$- and $\tI(2r)$-subgroups. Such a path must be even as shown above. On the other hand, as we are assuming that $\mathrm{cc}_2(W_\Gamma)$ attains the upper bound, the vertex set of this path---being a subset of $V(\Gamma)$---would itself have to yield a parabolic listed in \Cref{fig:Tabelle}, which is impossible for the following reason:  

Suppose there exists such an edge path $\gamma$ in $\Gamma$ connecting the $\tA_1$-vertex to a vertex in the $\tI(2r)$. There is at least one extra vertex $v$ lying on the path $\gamma$ connected to (at least) one vertex in the subdiagram $\tI(2r)$. This connecting edge is even and has a finite label. Thus $\{v\}\cup \tI(2r)$ yields an irreducible triangle subgroup of $W_\Gamma$ of the form $\Delta(2p,2q,2r)$ with $r\geq 2$, $p \geq 2$, $q \geq 1$. (See \Cref{sec:triangle} for more on triangle groups.) To reach the upper bound, every $k$-element subset of $V(\Gamma)$ must yield a parabolic defining a vertex in the $k$-odd graph. Hence this triangle subgroup $\Delta(2p,2q,2r)$ we just constructed must yield a vertex in $\Gamma^3$. So in particular it has to be finite and listed in \Cref{fig:Tabelle}. But such a group does not exist. Hence the path $\gamma$ connecting $\tA_1$ to $\tI(2r)$ cannot exist.

The non-existence of the connecting path implies that the diagram $\Gamma$ is disconnected and decomposes as $\Gamma = \Gamma_1 \sqcup \Gamma_2$ with  both $\Gamma_1, \Gamma_2$ nonempty. Again because $\mathrm{cc}_2(W)$ is maximal, the proper parabolic subgroups $W_{\Gamma_1}$ and $W_{\Gamma_2}$ now need to be spherical. But then $W = W_{\Gamma_1} \times W_{\Gamma_2}$ itself would be finite which contradicts our standing assumption. Thus the rank of $W_\Gamma$ is at most three, as claimed.

We now deduce the bounds from Inequality~(1). The lower bound is a triviality. Indeed, recall that $V(\Omega^k) = V(\Gamma^k)$. By \Cref{def:odd-graphs}, two vertices $W_J, W_L \in V(\Gamma^k)$ are connected by an edge when $W_J$ and $W_L$ are odd-adjacent. Again by \Cref{def:odd-graphs}, this implies in particular that $W_J$ and $W_L$ are isomorphic. Thus there is an edge in $\Omega^k$ connecting $W_J$ to $W_L$; cf. \Cref{def:omegas}. That is, any two vertices adjacent in $\Gamma^k$ are automatically adjacent in $\Omega^k$. This implies that 
\[\mathrm{cc}_2(W_\Gamma) = |\pi_0(\Gamma^1)| + \ldots + |\pi_0(\Gamma^n)| \geq |\pi_0(\Omega^1)|+\ldots+|\pi_0(\Omega^n)|.\]

Now let $W_{\Delta_1},\ldots, W_{\Delta_m}$ denote the maximal spherical standard parabolic subgroups of $W_\Gamma$ with respect to set-theoretic containment. If $w\in W_\Gamma$ is an involution, apply \Cref{lem:LongestEltsRepresentInvolutions} to find a standard spherical parabolic $W_J$ that contains a conjugate of $w$. Being spherical, this parabolic is, in turn, contained in a $W_{\Delta_i}$ for some $i \in \{1,\ldots,m\}$. This shows that
\[
\cc(W_\Gamma)\leq \cc(W_{\Delta_1})+\ldots +\cc(W_{\Delta_m}),
\]
as desired. 
\end{proof}

We remark that the upper bound from Inequality~(1) in \Cref{thm:obereSchranke} is sharp for finite Coxeter groups for trivial, uninteresting reasons: since they are spherical parabolic subgroups of themselves, the given inequality becomes $\cc(W_\Gamma) \leq \cc(W_{\Delta_1})$ where $\Delta_1=\Gamma$ itself. For a curious question about the lower bound with the $\mc{O}$-graphs $\Omega^k$, see \Cref{obs:omegas}. The following example gives a spherical nonabelian group that attains the numerical upper bound for $\cc$ from Inequality~(2) and shows that the lower bound from Inequality~(1) in \Cref{thm:obereSchranke} with $\mc{O}$-graphs is strict in general. 

\begin{example}[A nonabelian finite group with maximal $\cc$] \label{ex:nonabelianupper}
Define $\Gamma$ as follows: if $n=2k$ is even, $\Gamma$ is the disjoint union of $k=n/2$ dihedral diagrams $\tI(2m_1)$, $\dots$, $\tI(2m_{k})$ with $m_i \geq 2$; in case $n=2k+1$ is odd, $\Gamma$ is the disjoint union of $\tI(2m_1)$, $\dots$, $\tI(2m_{k})$ with one $\tA_1$, again with $m_i \geq 2$. Thus $W_\Gamma$ is nonabelian by construction, and since every $i$-element subset of $V(\Gamma)$ gives a parabolic subgroup whose irreducible components are listed in \Cref{fig:Tabelle}, its $i$-odd graph $\Gamma^i$ has $\binom{n}{i}$ vertices. As $\Gamma$ is even, no pair of its rank-$i$ parabolics is odd-adjacent. Whence every $\Gamma^i$ is totally disconnected and $\mathrm{cc}_2\left(W_\Gamma\right) =\sum_{i=1}^{n}\binom{n}{i}=2^n-1$. 

Note that, taking the indices $m_i$ to be all distinct, none of the factors $\tI(2m_i)$ are pairwise isomorphic. Hence $\cc(W_\Gamma) > \sum_{i=1}^n |\pi_0(\Omega^i)|$ is a strict inequality in this case.
\end{example}

We finish this section by deducing the additive properties for the number of conjugacy classes of elements of finite order in free and direct products. 

\begin{proof}[Proof of \Cref{thm:FreeDirectProducts}]
Let $m\geq 2$ and $G, H$ be groups such that $\ccm(G)<\infty$ and $\ccm(H)<\infty$.

Let $w\in G\ast H$ be an element of order $m$. Using the action of $G\ast H$ on the associated Bass--Serre tree $T_{G\ast T}$ one shows that $w$ is contained in a conjugate of $G$ or of $H$, see \cite[I{\S}6, Proposition~21]{SerreTrees}. Moreover, $w$ cannot be simultaneously conjugated to an element of $G$ and an element of $H$, since otherwise $w$ would fix an edge in $T_{G \ast H}$ which is not possible. Hence 
\[
\ccm(G\ast H)=\ccm(G) + \ccm(H).\]

Now, let $w\in G\times H$ be an element of order $m$. Then there exist $g\in G$, $h\in H$ such that $w=(g,h)$. Assume first that $g=1$ (resp. $h=1$). An easy calculation shows in these cases that $w$ is conjugated to an element in $H$ (resp. in $G$) of order $m$. Note that, since $G$ commutes with $H$, an element of $G$ cannot be conjugated to an element of $H$.

If $g\neq 1$ and $h\neq 1$, then $\ord(g)=a, \ord(h)=b$ and the least common multiple of $a$ and $b$, denoted by $\lcm(a,b)$, is equal to $m$. Let $(x, y)\in G\times H$. Then 
\[
(x, y)(g,h)(x, y)^{-1}=(xgx^{-1}, yhy^{-1}). 
\]
Hence, $(g,h)$ can only be conjugated to an element in $G\setminus\left\{1\right\}\times H\setminus\left\{1\right\}$. Further, two elements $(g_1, h_1), (g_2, h_2)\in G\setminus\left\{1\right\}\times H\setminus\left\{1\right\}$ with $\ord(g_1,h_1)=\ord(g_2, h_2)=m$  are conjugated if and only if $\ord(g_1)=\ord(g_2)=a$, $\ord(h_1)=\ord(h_2)=b$ and $g_1$ is conjugated to $g_2$ and $h_1$ is conjugated to $h_2$ and $\lcm(a,b)=m$. This shows  
\[
\ccm(G\times H)=\underset{\lcm(k,l)=m}{\sum_{(k,l)\in\N\times\N}} \cck(G)\cdot \ccl(H),
\]
and the expression for $m=p$ a prime is immediate.
\end{proof}

%%%%%%%%%%%%%%%%%%%%%%%%%%%%%%%%%%%%%%%%%%%%%%%%%%%%%%%%%%%%%%
\section{Applications to some subclasses of Coxeter groups}
\label{sec:specialCases}

To provide extended examples we compute in \Cref{sec:triangle} the number of conjugacy classes of involutions in any given triangle group  and discuss in \Cref{sec:RACS}  the subclass of right-angled Coxeter groups. 
In \Cref{sec:finite-affine-Coxeter} we determine $\cc$ in some of the spherical and affine cases.

%%%%%%%%%%%%%%
\subsection{Triangle Coxeter groups} 
\label{sec:triangle}

By \Cref{thm:obereSchranke}, the family of triangle groups witnesses sharpness of upper bounds for $\mathrm{cc}_2$. Here we explicitly record the number of conjugacy classes of involutions for all infinite triangle groups. Let $\Delta$ be the complete graph  with vertex set $\left\{s_1, s_2, s_3\right\}$ and an edge-labeling $p,q,r\in\mathbb{N}_{\geq 2}\cup\left\{\infty\right\}$, see \Cref{fig:Deltapqr}. By definition, the assosiated \emph{triangle group} $\Delta(p,q,r)$ is defined via the presentation 
\[\Delta(p,q,r):=\langle s_1, s_2, s_3\mid s_1^2, s_2^2, s_3^2, (s_1s_2)^p, (s_1s_3)^q, (s_2s_3)^r\rangle.\] 
We remark that $|\Delta(p,q,r)|=\infty$ if and only if at least one edge-label is $\infty$ or $\frac{1}{p}+\frac{1}{q}+\frac{1}{r}\leq 1$; see \cite[Chapter 6]{Davis2008}. 

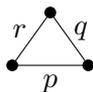
\begin{figure}[htb]
	\begin{center}
	\captionsetup{justification=centering}
		\begin{tikzpicture}
			\draw[fill=black]  (0,0) circle (2pt);
            \draw[fill=black]  (1,0) circle (2pt);
            \draw[fill=black]  (0.5,0.7) circle (2pt);
            \draw (0,0)--(1,0);
            \draw (1,0)--(0.5, 0.7);
            \draw (0,0)--(0.5, 0.7);
            \node at (0.5, -0.25) {$p$};
            \node at (0.1, 0.45) {$r$};
            \node at (0.9, 0.45) {$q$};
        \end{tikzpicture}
        \caption{Defining diagram $\Delta$ for a triangle group. (Note that this might differ from a Coxeter diagram.)}
	      \label{fig:Deltapqr}
    \end{center}
\end{figure}

Using \Cref{thm:CardConjClass}, we quickly compute $\mathrm{cc}_2(\Delta(p,q,r))$, the number of conjugacy classes of involutions, for all infinite triangle groups $\Delta(p,q,r)$. Since the number of conjugacy clases of reflections is given by $\vert\pi_0(\Gamma^1)\vert$ (cf. \cite[Lemma 3.6]{BMMN}), our findings give quick means to check the discrepancy between the (conjugacy classes of) reflections and involutions of $\Delta(p,q,r)$.  
For $p, q, r\in\mathbb{N}$ we have
\begin{align*}
\mathrm{cc}_2(\Delta(p,q,r)) & =
\begin{cases}
6, & \text{if }p, q, r\text { are even},\\
1, & \text{if }p, q, r\text { are odd},\\
4, &  \text{if } \text{ precisely one edge label is odd},\\
2, & \text{if } \text{ precisely one edge label is even};
\end{cases} \\
&\\
\mathrm{cc}_2(\Delta(\infty,q, r)) &=
\begin{cases}
5, & \text{if }q\text{ and }r\text{ are even},\\
1, &  \text{if }q\text{ and }r\text{ are odd},\\
3, & \text{if } (q+r) \text{ is odd};
\end{cases} \\
&\\
\mathrm{cc}_2(\Delta(\infty,\infty, r)) &=
\begin{cases}
4, & \text{if }r\text{ is even},\\
2, & \text{if }r\text{ is odd}.
\end{cases}
\end{align*}
Finally, $\mathrm{cc}_2(\Delta(\infty,\infty,\infty)) =3$ by \Cref{ex:universalCoxetergroups}, which completes the list. It is curious how all possible values from $1$ to $6$ are realised as the $\cc$ of rank-$3$ groups. We have no structural explanation for this phenomenon at the moment. 

\begin{remark}[When do $\mc{O}$-graphs suffice to compute $\cc$?] \label{obs:omegas}
    Using $\mc{O}$-graphs (cf. \cref{def:omegas}), if $\Delta(p,q,r)$ is infinite, then 
$\mathrm{cc}_2(\Delta(p,q,r))=|\pi_0(\Omega^1)|+\vert\pi_0(\Omega^2)\vert$ if and only if the even edge labels in the defining triangle are pairwise distinct.

Hence we raise the following question: Is it possible to characterize those Coxeter groups $W_\Gamma$ where the inequality $\cc(W_\Gamma) \geq \sum_{k=1}^{|V(\Gamma)|} |\pi_0(\Omega^k)|$ is indeed an equality?
\end{remark}

%%%%%%%%%%%%%%%%%%%%%%%%%%%%%%%%%%%%%%%%%%%%%%%%%%
\subsection{Right-angled Coxeter groups}
\label{sec:RACS}

RACGs are typically constructed using a differnt \emph{presentation diagram} (or \emph{defining graph}), which we now introduce. Let $\Lambda$ be a finite simplicial graph with vertex set $V(\Lambda)$ and edge set $E(\Lambda)$. The right-angled Coxeter group $W(\Lambda)$ is defined as follows:
\[
W(\Lambda) = \langle V(\Lambda)\mid v^2\text{ for all }v\in V(\Lambda), (vw)^2\text{ for all } \left\{v,w\right\}\in E(\Lambda)\rangle.
\]

Note that, for many RACGs, the number $\cc$ of conjugacy classes of involutions can be computed by first looking at decompositions, applying \Cref{thm:FreeDirectProducts} to reduce complexity, then invoking our main \Cref{thm:CardConjClass}. Here we offer alternative means to determine $\cc$ for such groups. 
The proof has a geometric flavour and for it we do not need the results from \Cref{sec:OddGraphs,sec:mainproofs}.

\begin{proposition}
Let $W(\Lambda)$ be a right-angled Coxeter group. The number $\mathrm{cc}_2(W(\Lambda))$ of conjugacy classes of involutions in $W(\Lambda)$ is equal to the number of nontrivial cliques (i.e., complete subgraphs of rank at least one) in the defining graph $\Lambda$.
\end{proposition}

\begin{proof}
Let $W(\Lambda)$ be a right-angled Coxeter group and let $X$ be the associated Davis--Moussong complex (cf. \cite[Chapters~7 and~12]{Davis2008}), on which $W(\Lambda)$ acts geometrically. We denote this action by $\Phi\colon W(\Lambda) \to \Isom(X)$. 

Let $w\in W(\Lambda)$ be an involution. By the Bruhat--Tits fixed point theorem \cite[Theorem~11.23]{AbramenkoBrown} we know that $\Fix(\Phi(w))$ is nonempty. Since the action $\Phi$ is type-preserving, there exists a vertex in $F:=\Fix(\Phi(w))$. Let then $y\in F$ be a vertex. By definition of $X$ we have $y=gW(\Delta)$, where $\Delta$ is a nontrivial clique in $\Lambda$. Further, $w\in gW(\Delta) g^{-1}$.

We define $A_w:=\bigcap_{g_i W(\Delta_i)\in F} g_i W(\Delta_i)g_i^{-1}$. Since $w \in g_i W(\Delta_i)g_i^{-1}$ for every $g_i W(\Delta_i)\in F$, the intersection $A_w$ is nontrivial. Moreover, $A_w$ is a parabolic subgroup since it is an intersection of parabolics; cf. \cite{Qi2007}. Thus $A_w = x W(\Delta') x^{-1}$ for some $x \in W$ and some nonempty clique $\Delta'$ in $\Lambda$. Hence $w \in A_w$ is conjugate to $v_1v_2\ldots v_r$ where $V(\Delta')=\left\{v_1,\ldots, v_r\right\}$.

Let now $\Delta_1$ and $\Delta_2$ be two nontrivial cliques. Let $w_1=x_1\ldots x_n$ and $w_2=y_1\ldots y_m$ be involutions where $V(\Delta_1)=\left\{x_1,\ldots, x_n\right\}$ and $V(\Delta_2)=\left\{y_1,\ldots, y_m\right\}$. 
If $w_1$ is conjugate to $w_2$, then the parabolic closure of $w_1$ is conjugate to the parabolic closure of $w_2$ by \cref{lem:ConjLongest}. . 
As the parabolic closure of $w_i$ is $W(\Delta_i)$, it follows from \cite[Corollary~3.8]{AntolinMinasyan} that $\Delta_1=\Delta_2$.
\end{proof}

%%%%%%%%%%%%%%%%%%%%%%%%%%%%%%%%%%%%%%%%%%%%%%%%%%
% subsection on finite and affine Coxeter groups

\subsection{Spherical and affine Coxeter groups}
\label{sec:finite-affine-Coxeter}

The irreducible finite and affine Coxeter groups have been classified in a short list of infinite families $\tA_n$, $\tB_n$, $\tD_n$, $\tI(m)$, $\atA_n$, $\atB_n$, $\atC_n$, $\atD_n$, shown in \Cref{fig:InfiniteFamiliesFiniteCoxeter,fig:InfiniteFamiliesAffineCoxeter}, and some exceptional types. 
In this section we provide formulae for $\cc$ for some of these families. We list the numbers in ranks $\leq 11$ in all affine types in \Cref{sec:app:tables-affine}. 
Note that types $\til{\mathtt{I}}_1$ and $\atG$ have already been dealt with in \cref{ex:universalCoxetergroups} and \Cref{sec:triangle}, respectively, since $W_{\til{\mathtt{I}}_1} \cong \Z/2\Z \ast \Z/2\Z$ and $W_{\atG} \cong \Delta(2,3,6)$.

The next theorem summarizes our findings for the irreducible infinite cases.

\begin{theorem}[Some affine Coxeter groups] 
\label{thm:affineCoxeter}
    The number of conjugacy classes of involutions of Coxeter groups of type $\atA_n$ is given by the formula
    \[
     \mathrm{cc}_2 (W_{\atA_n})= 
        \begin{cases}
            k,  &\text{ if } n=2k,\\
            k+2, &\text{ if } n=2k+1;
        \end{cases}
    \] 
    and in type $\atC_n$ by the formula
    \[
        \cc (W_{\atC_n}) = 
        \begin{cases}
            4k^2+2k, &\text{ if } n=2k,\\
            4k^2+6k+2, &\text{ if } n=2k+1.\\
       \end{cases}
    \] 
\end{theorem}
\begin{proof}
Our formula for type $\atA_n$ follows immediately from the more general \cref{prop:odd-circle}, to be proved below. 
The formula for type $\atC_n$ is obtained using entirely analogous arguments, drawing from techniques developed in \cref{sec:OddGraphs}, whence we omit the proof. 
\end{proof}

\begin{figure}[htb]
	\begin{center}
	\captionsetup{justification=centering}
		\begin{tikzpicture}[scale=1, transform shape]
            \draw[fill=black] (0,-1.5) circle (2pt);
            \draw[fill=black] (1,-1.5) circle (2pt);
            \draw (0,-1.5)--(1,-1.5);
            \draw[fill=black] (2,-1.5) circle (2pt);
            \draw[fill=black] (3,-1.5) circle (2pt);
            \draw[dashed] (1,-1.5)--(2, -1.5);
            \draw (2,-1.5)--(3,-1.5);
            \draw[fill=black] (1.5, -0.8) circle (2pt);
            \draw (0,-1.5)--(1.5, -0.8);
            \draw (3,-1.5)--(1.5, -0.8);
            \node at (0, -1.8) {$s_0$};
            \node at (1, -1.8) {$s_1$};
            \node at (2, -1.8) {$s_{n-2}$};
            \node at (3, -1.8) {$s_{n-1}$};
            \node at (1.5, -0.5) {$s_n$};
            \node at (-0.9, -1.5) {$\underset{n\geq2}{\atA_{n}}$};	

            \draw[fill=black] (7,-1.5) circle (2pt);
            \draw[fill=black] (8,-1.5) circle (2pt);
            \node at (7.5,-1.3) {$4$};
            \draw (7,-1.5)--(8,-1.5);
            \draw (8,-1.5)--(9,-1.5);
            \draw[dashed] (9,-1.5)--(10,-1.5);
            \draw[fill=black] (10,-1.5) circle (2pt);
            \draw[fill=black] (9,-1.5) circle (2pt);
            \draw[fill=black] (11, -0.8) circle (2pt);
            \draw[fill=black] (11, -2.2) circle (2pt);
            \draw (10, -1.5)--(11,-0.8);
            \draw (10,-1.5)--(11, -2.2);
            \node at (6.1, -1.5) {$\underset{n\geq3}{\atB_{n}}$};	

            \draw[fill=black] (0,-3.7) circle (2pt);
            \draw[fill=black] (1,-3.7) circle (2pt);
            \draw (0,-3.7)--(1, -3.7);
            \node at (0.5, -3.5) {$4$};
            \draw[fill=black] (2, -3.7) circle (2pt);
            \draw (1, -3.7)--(2,-3.7);
            \draw[dashed] (2,-3.7)--(3,-3.7);
            \draw[fill=black] (3,-3.7) circle (2pt);
            \draw[fill=black] (4,-3.7) circle (2pt);
            \draw (3,-3.7)--(4,-3.7);
            \node at (3.5, -3.5) {$4$};
            \node at (-0.9, -3.7) {$\underset{n\geq2}{\atC_{n}}$};

            \draw[fill=black] (7,-3) circle (2pt);
            \draw[fill=black] (7,-4.4) circle (2pt);
            \draw[fill=black] (8,-3.7) circle (2pt);
            \draw (7,-3)--(8,-3.7);
            \draw (7,-4.4)--(8,-3.7);
            \draw[fill=black] (9,-3.7) circle (2pt);
            \draw (8,-3.7)--(9,-3.7);
            \draw[dashed] (9,-3.7)--(10,-3.7);
            \draw[fill=black] (10,-3.7) circle (2pt);
            \draw[fill=black] (11, -3) circle (2pt);
            \draw[fill=black] (11, -4.4) circle (2pt);
            \draw (10,-3.7)--(11, -3);
            \draw (10,-3.7)--(11, -4.4);           
            \node at (6.1, -3.7) {$\underset{n\geq4}{\atD_{n}}$};
\end{tikzpicture}
	\caption{Infinite families of affine irreducible Coxeter diagrams. 
}
     \label{fig:InfiniteFamiliesAffineCoxeter}
	\end{center}
\end{figure}
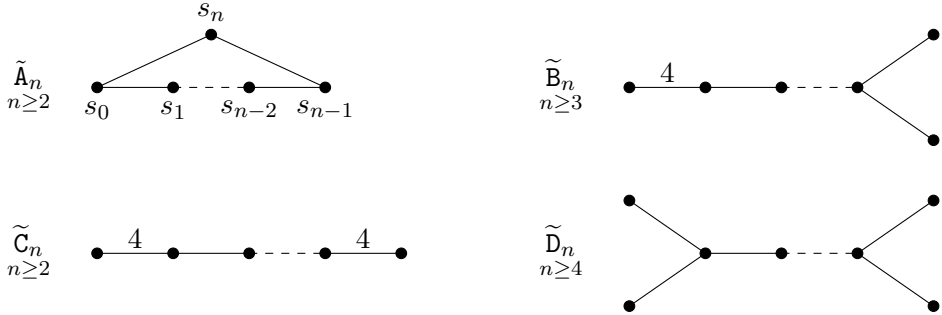

\begin{remark} \label{obs:complicated}
    The combinatorial complexity of counting connected components in $\Gamma^k$ for a diagram $\Gamma$ of type $\atC_n$ is increased by the existence of subdiagrams of types 
    \[
    \tB_{k-i} \sqcup \left( \underbrace{\tA_1 \sqcup \cdots \sqcup \tA_1}_{i \text{ copies}} \right)
    \]
    in $\Gamma$, with $i \in \lbrace 0, \dots, k \rbrace$. Nevertheless, a combination of the tools presented in \cref{sec:OddGraphs} still suffices to obtain the formula for $\cc (W_{\atC_n})$ presented above.
    The reader interested in detailed proofs is referred to the first author's forthcoming thesis,  \cite{ThesisAnna}.
\end{remark}

\Cref{prop:odd-circle} provides a formula for the number of conjugacy classes of involutions for all Coxeter groups whose Coxeter diagram is an odd-labeled circle, provided it contains no subdiagram of type $\tH$. 
Note that this class includes groups of type $\atA_n$. 
The relevant diagrams are shown in \Cref{fig:Circle} in \Cref{sec:OddGraphs}.

\begin{proposition}[Computing $\cc$ in odd circles]
\label{prop:odd-circle}
    Let  $n \geq 2$ and $\Gamma$ be a Coxeter diagram as shown in \Cref{fig:Circle}, with 
    $m_{i,i+1}$ unlabeled or odd for all $i=0,\ldots, n-1$, and assuming $\Gamma$ contains no subgraph isomorphic to $\tH_3$. Then 
    \[
    \mathrm{cc}_2 (W_{\Gamma})= 
        \begin{cases}
            \ell,  &\text{ if } n=2\ell,\\
            \ell+2, &\text{ if } n=2\ell+1.
        \end{cases}
    \]
\end{proposition}

\begin{proof}

    \cref{lemma:VerticesInTypeA} implies that, for $k\leq \ell$ the graph $\Gamma^k$ has no isolated vertices. 
    To determine the connected components of $\Gamma^k$ in this case, the same argument as in the proof of \cref{lemma:VerticesInTypeA_path} works, showing that $\Gamma^k$ is always connected for $k \leq \ell$.
    (Note that here the distance between indices of generators of a given parabolic in $\Gamma^k$ must also be at least $2$, but now indices are taken $\mod{n+1}$.)

Since $\Gamma$ is a circle, taking $l>\ell$ gives $\Gamma^l= \leer$ for $n=2\ell$; in turn, for $n=2\ell+1$, the graph $\Gamma^{\ell+1}$ contains precisely two isolated vertices, namely $\langle s_0, s_2, s_4, \dots, s_{2\ell}\rangle$ and $\langle s_1, s_3, s_5, \dots, s_{2\ell+1}\rangle$. Finally, $\Gamma^l$ for $l > \ell+1$ is also empty. This completes the proof.
\end{proof}

\begin{remark}[Affine types $\atB$ and $\atD$] 
In Coxeter diagrams of types $\atB$ and $\atD$, various combinations of diagrams from \Cref{fig:Tabelle} can occur as subgraphs, leading to a more complicated structure of the connected components of each $\Gamma^k$. Particularly in type $\atD$, the difficulty in counting conjugacy classes of involutions is also observed in recent work of Mili\'cevi\'c, Schwer, and Thomas \cite{MST5}. 

While we were unable to identify a closed formula for the total number of conjugacy classes of involutions $\cc$ in these cases, there seem to be some interesting patterns for the numbers of connected components of their odd graphs $\Gamma^k$. That is, as $n=|V(\Gamma)|$ grows, the sequence $(|\pi_0(\Gamma^k)|)_{k \in \N}$ might have interesting combinatorial properties in both types $\atB_n$ and $\atD_n$. 

For small ranks the number of conjugacy classes of involutions in these groups is shown in \Cref{table:atB,table:atD} in \Cref{sec:app:tables-affine}. Again, these were obtained by direct inspection or making use of some tools from \Cref{sec:OddGraphs}. 
\end{remark}

\begin{remark}[Remaining affine types $\atE_n$, $\atF$, $\atG$, and $\atI$]
One checks directly with \Cref{thm:CardConjClass} that $\cc(W_{\atE_6})=5$, $\cc(W_{\atE_7})=19$, $\cc(W_{\atE_8})=14$ and $\cc(W_{\atF}) = 12$. Types $\atG$ and $\atI$ were covered in \Cref{sec:triangle} and by \cref{thm:FreeDirectProducts}, respectively. In addition we refer the reader to the tables in \Cref{sec:app:tables-affine} for numbers in affine types of ranks $\leq 11$. We note in passing that $\cc(W_{\atE_7})$ has been previously computed explicitly by Richardson \cite[Example~3.4]{Richardson_1982}---we remark that Richardson found $20$ conjugacy classes as he also considers the trivial element to be an involution.
\end{remark}

The formulae for spherical types $\tA_n$ and $\tB_n=\tC_n$ are computed using similar methods, whence proofs are omitted. We record them in the next theorem. 

\begin{theorem}[Some spherical Coxeter groups] 
\label{thm:finiteCoxeter}
    The number of conjugacy classes of involutions is given in type $\tA_n$ by the formula
    \[
     \mathrm{cc}_2 (W_{\tA_n})= 
        \begin{cases}
            k,  &\text{ if } n=2k,\\
            k+1, &\text{ if } n=2k+1;
        \end{cases}
    \] 
    and in type $\tC_n$ by the formula
    \[
        \cc (W_{\tC_n}) = \begin{cases}
            k^2+2k, &\text{ if } n=2k,\\
            k^2+3k+1, &\text{ if } n=2k+1.\\
        \end{cases}
    \] 
\end{theorem}

Note that the formula for type $\tA$ is a straightforward corollary to \Cref{lemma:VerticesInTypeA_path} (after \Cref{thm:CardConjClass}). Type $\tC$ is dealt with analogously, taking a bit more care with the combinatorics similarly to the discussion in \Cref{obs:complicated}. 
A detailed proof of \Cref{thm:finiteCoxeter} will be avaliable in the first author's thesis \cite{ThesisAnna}.

\begin{figure}[h]
	\begin{center}
	\captionsetup{justification=centering}
		\begin{tikzpicture}
			\draw[fill=black]  (0,0) circle (2pt);
			\draw[fill=black]  (1,0) circle (2pt);
            \draw[fill=black]  (2,0) circle (2pt);
			\draw (0,0)--(1,0);
            \draw[dashed] (1,0)--(2,0);
            \draw[fill=black] (3,0) circle (2pt);
            \draw (2,0)--(3,0);
			\node at (-0.9,0) {$\underset{n\geq 1}{\tA_{n}}$};

            \draw[fill=black]  (6,0) circle (2pt);
			\draw[fill=black]  (7,0) circle (2pt);
            \draw[fill=black]  (8,0) circle (2pt);
			\draw (6,0)--(7,0);
            \draw[dashed] (7,0)--(8,0);
            \draw[fill=black] (9,0) circle (2pt);
            \draw (8,0)--(9,0);
            \node at (8.5, 0.2) {$4$};
            \node at (5.1,0) {$\underset{n\geq 2}{\tB_{n}=\tC_n}$};

            \draw[fill=black]  (0,-1.5) circle (2pt);
			\draw[fill=black]  (1,-1.5) circle (2pt);
            \draw[fill=black]  (2,-1.5) circle (2pt);
			\draw (0,-1.5)--(1,-1.5);
            \draw[dashed] (1,-1.5)--(2,-1.5);
            \draw[fill=black] (3,-0.8) circle (2pt);
            \draw[fill=black] (3,-2.2) circle (2pt);
            \draw (2,-1.5)--(3, -0.8);
            \draw (2,-1.5)--(3, -2.2);
            \node at (-0.9,-1.5) {$\underset{n\geq 4}{\tD_{n}}$};

            \draw[fill=black] (7,-1.5) circle (2pt);
            \draw[fill=black] (8,-1.5) circle (2pt);
            \draw (7,-1.5)--(8,-1.5);
            \node at (7.5, -1.3){$m$};
            \node at (6.1,-1.5) {$\underset{m\geq 4}{\tI(m)}$};
		\end{tikzpicture}
	\caption{Infinite families of finite irreducible Coxeter diagrams.}
    \label{fig:InfiniteFamiliesFiniteCoxeter}
	\end{center}
\end{figure}
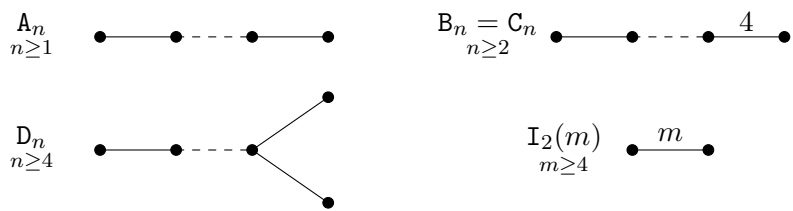

%%%%%%%%%%%%%%%%%%%%%%%%%%%%%%%%%%%%%%%%%%%%%%%%%%
%section: tables

\section{Tables listing classes in affine Coxeter groups}
\label{sec:app:tables-affine}

The tables in this section list the number of conjugacy classes of involutions for all affine irreducible Coxeter groups of rank $n+1 \leq 11$, with the exception of type $\atA$ which displays linear values for $\cc$, as seen in \Cref{thm:affineCoxeter}.  
Recall that $\Gamma^{n+1}$ is empty for such systems as $W_\Gamma$ is infinite. 
In each table, the summands in the $j$-th line correspond, in order, to the number of connected components of the graphs $\Gamma^k$ for $k\leq j$.

\medskip

\begin{table}[htb]
\small
    \begin{tabular}{p{1cm} p{8cm}}
     Type & $\cc(W_{\atB_n}) = \left|\pi_0 (\Gamma^1)\right|+\ldots + \left| \pi_0 (\Gamma^{n})\right|$\\
     \midrule
     \midrule
     $\atB_3$&  $2+4+3=9$  \\
     \midrule
     $\atB_4$&  $2+5+5+2=14$  \\
     \midrule
     $\atB_5$&  $2+4+6+5+2=19$  \\
     \midrule
     $\atB_6$&   $2+4+6+6+5+2=25$ \\
     \midrule
     $\atB_7$&  $2+4+6+7+6+5+2=32$   \\
     \midrule
     $\atB_8$&   $2+4+6+9+8+6+5+2=42$ \\
     \midrule
     $\atB_9$& $2+4+6+8+10+9+6+5+2=52$   \\
     \midrule
     $\atB_{10}$&   $2+4+6+8+11+11+9+6+5+2=64$  \\
     \bottomrule
\end{tabular}
\caption{Conjugacy classes of involutions in type $\atB$.}
\label{table:atB}
\end{table}

\medskip

\begin{table}[htb]
\small
    \begin{tabular}{p{1cm} p{8cm}}
     Type & $\cc(W_{\atC_n}) = \left|\pi_0 (\Gamma^1)\right|+\ldots + \left| \pi_0 (\Gamma^{n})\right|$\\
     \midrule
     \midrule
     $\atC_2$ & $3+3=6$ \\
     \midrule
     $\atC_3$& $3+5+3=11$  \\
     \midrule
     $\atC_4$&  $3+6+7+4=20$  \\
     \midrule
     $\atC_5$&  $3+6+9+8+4=30$  \\
     \midrule
     $\atC_6$&  $3+6+10+11+8+4=42$  \\
     \midrule
     $\atC_7$&  $3+6+10+13+12+8+4=56$   \\
     \midrule
     $\atC_8$&  $3+6+10+14+15+12+8+4=72$  \\
     \midrule
     $\atC_9$&  $3+6+10+14+17+16+12+8+4=90$   \\
     \midrule
     $\atC_{10}$&  $3+6+10+14+18+19+16+12+8+4=110$  \\
     \bottomrule
\end{tabular}
\caption{Conjugacy classes of involutions in type $\atC$.}
\label{table:atC}
\end{table}

\medskip

\begin{table}[htb]
\small
    \begin{tabular}{p{1cm} p{8cm}}
    Type & $\cc(W_{\atD_n}) = \left|\pi_0 (\Gamma^1)\right|+\ldots + \left| \pi_0 (\Gamma^{n})\right|$\\
     \midrule
     \midrule
     $\atD_4$& $1+6+4+5=16$   \\
     \midrule
     $\atD_5$& $1+3+2+3+0=9$   \\
     \midrule
     $\atD_6$&  $1+3+6+7+5+4=26$  \\
     \midrule
     $\atD_7$&  $1+3+3+5+3+4+0=19$   \\
     \midrule
     $\atD_8$& $1+3+3+9+7+9+6+7=45$  \\
     \midrule
     $\atD_9$&  $1+3+3+6+5+7+4+5+0=34$ \\
     \midrule
     $\atD_{10}$&  $1+3+3+6+9+11+9+9+7+4=62$ \\
     \bottomrule
\end{tabular}
\caption{Conjugacy classes of involutions in type $\atD$.} 
\label{table:atD}
\end{table}

\medskip

\begin{table}[h]
\small
\begin{tabular}{p{1cm} p{8cm}}
%    \toprule
      Type & $\cc(W_{\atX_n})= \left|\pi_0 (\Gamma^1)\right|+\ldots + \left| \pi_0 (\Gamma^{n})\right|$\\
     \midrule
     \midrule
     $\atE_6$ & $1+1+1+2+0+0=5$ \\
     \midrule
     $\atE_7$ & $1+1+3+4+3+3+4=19$ \\
     \midrule
     $\atE_8$& $1+1+1+3+2+2+1+3=14$ \\
     \midrule
     $\atF$&  $2+3+4+3=12$  \\
     \midrule
     $\atG$&  $2+2=4$  \\
     \midrule
     $\til{\mathtt{I}}_1$&  $2$ \\
     \bottomrule
\end{tabular}
\caption{Number of classes in the exceptional affine types.}
\label{table:exceptionalTypes}
\end{table}

\vspace{-8ex}

{\ }\newline
\vspace{10ex}
{\ }\newline
\vspace{10ex}
\printbibliography

\end{document}